%% file: main.tex
\documentclass[11pt]{article}

\usepackage[top=2.5cm,bottom=2.5cm,right=2.5cm,left=2.5cm]{geometry}

\input{commands}

\title{Minimal Trails in Restricted DAGs}

\author{Alexis Derumigny\thanks{Department of Applied Mathematics, Delft University of Technology, Delft, The Netherlands.
E-mail address: a.f.f.derumigny@tudelft.nl
},
Niels Horsman\thanks{Department of Applied Mathematics, Delft University of Technology, Delft, The Netherlands},
Dorota Kurowicka\thanks{Department of Applied Mathematics, Delft University of Technology, Delft, The Netherlands.
E-mail address: d.kurowicka@tudelft.nl
}
}

\date{\today}

\begin{document}

\maketitle

\begin{abstract}
In this paper, the properties of minimal trails in a directed acyclic graph that is restricted not to contain an active cycle are studied.
We are motivated by an application of the results in the copula-based Bayesian Network model developed recently.
We propose a partial order on the set of trails activated by a certain subset of nodes, and show that every minimal trail, according to such an order, has a simple structure.

\medskip


\end{abstract}


\input{Introduction}

\input{Digraphs}

\input{Ordered_Trails}

\input{No_Converging}

\input{With_converging}

\input{Conclusion}

\bibliographystyle{abbrv}
\bibliography{main}

\end{document}

%% file: commands.tex
\usepackage[hidelinks]{hyperref}
\usepackage{amsmath, amsfonts, amssymb, amsthm}
\usepackage{bbm} 
\usepackage{mathrsfs}
\usepackage{multicol}
\usepackage{enumitem}
\usepackage{tabularx}
\usepackage[
   nameinlink,
   noabbrev,
   capitalise
   ]{cleveref}
\usepackage{caption} 
\usepackage{subcaption}
\usepackage{graphicx,float}
\usepackage{thmtools}
\usepackage{enumitem}
\usepackage{cancel}

\setlist[enumerate]{itemsep=0.2em, topsep = 0em}
\setlist[itemize]{itemsep=0.2em, topsep = 0em}

\newtheorem{theorem}{Theorem}[section]
\newtheorem{proposition}{Proposition}[section]
\newtheorem{corollary}{Corollary}[section]
\newtheorem{lemma}{Lemma}[section]

\newtheorem{remark}{Remark}[section]

\newtheorem{definition}{Definition}[section]

\crefname{lemma}{Lemma}{Lemmas}
\Crefname{lemma}{Lemma}{Lemmas}
\crefname{thm}{Theorem}{Theorems}
\Crefname{thm}{Theorem}{Theorems}
\crefname{corollary}{Corollary}{Corollaries}
\Crefname{corollary}{Corollary}{Corollaries}

\crefname{thm}{Theorem}{Theorems}
\crefname{lem}{Lemma}{Lemmas}
\Crefname{thm}{Theorem}{Theorems}
\Crefname{lem}{Lemma}{Lemmas}
\crefname{cor}{Corollary}{Corollaries}
\Crefname{cor}{Corollary}{Corollaries}

\usepackage{tikz}
\usepackage{tkz-graph} 

\usetikzlibrary{graphs}
\usetikzlibrary{arrows.meta, chains, positioning, shapes.geometric}
\usetikzlibrary{graphs,quotes}
\usetikzlibrary{decorations.pathreplacing}
\usetikzlibrary{calc}
\usetikzlibrary{positioning}

\definecolor{myred}{rgb}{1.0, 0.22, 0.0}
\definecolor{mygreen}{rgb}{0.4, 1.0, 0.0}
\definecolor{amber}{rgb}{1.0, 0.75, 0.0}

\newcommand{\pgfextractangle}[3]{%
    \pgfmathanglebetweenpoints{\pgfpointanchor{#2}{center}}
                              {\pgfpointanchor{#3}{center}}
    \global\let#1\pgfmathresult  
    }

\newcommand{\TikzHar}[4]{
    \pgfextractangle{\angle}{#1}{#2}
        \draw[transform canvas={yshift= cos(\angle)*#3, xshift= -sin(\angle)*#3},-left to, line width = #4] (#1) -- (#2);
    \draw[transform canvas={yshift= -cos(\angle)*#3, xshift= sin(\angle)*#3},left to-, line width = #4] (#1) -- (#2);
}

\newcommand{\har}{\rightleftharpoons} 

\newcommand{\G}{{G}} 
\newcommand{\V}{{V}} 
\newcommand{\E}{{E}} 

\newcommand{\ra}{\rightarrow}
\newcommand{\la}{\leftarrow}

\newcommand{\pa}[1]{{pa}(#1)}
\newcommand{\ch}[1]{{ch}(#1)}

\newcommand{\an}[1]{{an}(#1)}
\newcommand{\de}[1]{{de}(#1)}

\newcommand{\dsepbig}[3]{{d\textrm{-}sep_{\G} \big(#1, #2 \, \big| \, #3 \big) }}

\newcommand{\notdsepbig}[3]{{\cancel{d\textrm{-}sep}_{\G} \big(#1, #2 \, \big| \, #3 \big) }}

\newcommand{\TRAILSbig}[3]{{{TRAILS} \big(#1, #2 \, \big| \, #3 \big) }}

\newcommand{\smallerTRAIL}{<_{TRAIL}}
\newcommand{\ConvCon}[1]{{{ConvCon}(#1) }}
\newcommand{\Cardi}[1]{{\big| #1 \big| }}
\newcommand{\cdotslong}{\cdots \cdots \cdots}
\newcommand{\prop}{\mathfrak{P}}

\newcommand{\sumlim}{\sum\limits}

\newcommand{\n}[2]{{n_{{#1}}({#2})}}

%% file: Introduction.tex
\section{Introduction}
The properties of directed acyclic graphs (DAGs) are extensively studied and found applications in many areas \cite{Bang_Jensen}.
A well-known fact about DAGs is that for these graphs a non-unique order of nodes, called well order, can always be found such that parents appear earlier in this order than their children.
This directionality is very useful as a notion of causality \cite{Pearl} or flow of information \cite{Ahujaetal}, but it also allows for an intuitive representation of a joint distribution of random variables that are assumed to correspond to nodes of a DAG.

A Bayesian Network (BN) is a graphical model for a set of random variables where the qualitative part is a DAG. 
This DAG represents in an intuitive way the relationships between the random variables that correspond to its nodes.
When two nodes in the DAG are connected by an arc, then the variables corresponding to these nodes are dependent.
If there is no arc between the nodes, then the corresponding random variables are either independent, or conditionally independent given some subset of variables (corresponding to a subset of nodes in the DAG).
Independence and conditional independence in a distribution represented by a BN can be read directly from the DAG by observing the structure of the graph, through the d-separation (\cref{def:graph_dsep}).
The concept of trails is crucial to define d-separation.
A trail $T$ from a node $x$ to a node $y$ is a path in the undirected graph $\overline{G}$ obtained from the DAG $G$ by removing directions.
$T$ is represented as
$$x \har v_1 \har \cdots \har v_n \har y,$$
where the symbol $\har$ corresponds to an arrow that can have one of possible directions in the DAG.
Given a set of nodes $Z$, a trail is either said to be blocked or to be activated by $Z$.
A trail $T$ can be blocked by $Z$ in two distinct ways:
1) if there is a node $v_i \in T$ such that in $G$ there is a serial or diverging connection at this node
($v_{i-1} \to v_i \to v_{i+1}$, $v_{i-1} \leftarrow v_i \rightarrow v_{i+1}$)
and $v_i\in Z$ or
2) if at $v_i$ there is a converging connection
($v_{i-1} \rightarrow v_i \leftarrow v_{i+1}$)
and neither $v_i$ nor any of its descendants are in $Z$.

If every trail from node $x$ to node $y$ is blocked by $Z$, then the random variables corresponding to $x$ and $y$ are conditionally independent given variables corresponding to nodes in $Z$ \cite{KollerFriedman}.

In this paper we study properties of trails activated by some set of nodes $Z$. Such a trail $T$ is of the form
\begin{equation}
\label{eq:full_trail}  
     x \har t^{0}_1 \har \cdots \har t^{0}_{\n{t}{0}} \ra 
    c_1 \la 
    \cdots \ra
    c_C \la t^{C}_1 \har \cdots \har t^{C}_{\n{t}{C}} \har y,
\end{equation}
where the nodes with converging connections; $c_1,\dots, c_C$, might be in  $Z$ (or their descendants are in $Z$ as is discussed in detail in \cref{sec:setoftrails}) and all other nodes on the trail are not in $Z$.
We are interested in the existence of certain arcs in $\G$ that can be deduced by observing properties of such trails. 

Our approach is to define a set of trails activated by $Z$ and equip this set with a partial order.
Then we are able to study minimal trails according to this order and show that such trails in a restricted DAG that does not contain a certain type of induced subgraph have nice properties (see \cref{thm:minimal_trails}).
The subgraph we will not allow is called an active cycle (\cref{def:active_cycle}) which is a cycle in the corresponding undirected graph that satisfy some conditions.
We will also introduce additional constraints on the types of relationships that we allow in the graph and consider how these extra constraints influence the properties of minimal trails
(see \cref{thm:minimal_trails_local_relationship}).

In \cref{sec:graphs} the necessary concepts that concern directed graphs and the d-separation are introduced.
Then in \cref{sec:setoftrails} the set of trails activated by a set $Z$ is defined and the partial order of elements in this set is presented.
We define sub-trails as the trails between the converging connections.
In the trail~\eqref{eq:full_trail} these are elements $t^{i}_1 \har \cdots \har t^{i}_{\n{t}{i}}$ where $i=1, \dots ,C$.
Such simple sub-trails are studied first in \cref{sec:proof_trail_with_no_converging}.
Moreover, since we want to use the results of this section in proofs concerning more general types of trails the results will be shown to hold also for trails that are minimal and contain elements in a subset of nodes, $K$.
The main results of the paper are contained in \cref{sec:proof:properties_trails_with_converging}.

%% file: Digraphs.tex
\section{Directed graphs}
\label{sec:graphs}

Let $\G = (V,E)$ be a directed graph with nodes $V$ and arcs $E$.
We only consider simple graphs without loops.
Moreover, let $\overline{\G}$ be the associated undirected graph called the \textbf{skeleton} of $G$, obtained from $G$ by removing the directions of the arcs.
A \textbf{path} is a sequence of nodes $(v_1,v_2,\dots,v_n)$ such that $\{v_1,v_2,\dots,v_n\}\subseteq \V$ and $\{(v_1,v_2), (v_2,v_3), \dots, (v_{n-1},v_n) \} \subseteq \E$ for some integer $n > 0$ called the \textbf{length} of the path.
A \textbf{trail} is an undirected path in $\overline{\G}$ which we will represent as $v_1 \har \cdots \har v_n$.
An arc between non-consecutive nodes in the trail is referred to as a \textbf{chord}.

A directed graph $G' = (V', E')$ is a \textbf{subgraph} of $G$ if $V'\subseteq V$, $E'\subseteq E$ and
for all arcs $w \rightarrow v \in E'$ the nodes $w$ and $v$ are in $V'$.
If $E'$ contains all arcs in $G$ between nodes in $V'$, then $G'$ is said to be \textbf{induced} by $V'$.

A path of the form $(v_1,v_2,\dots,v_n,v_1)$ is called a \textbf{cycle}.
We call $\G$ \textbf{acyclic} if it does not contain any cycle.

For each arc $w \rightarrow v \in \E$ the node $w$ is said to be the \textbf{parent} of $v$ and $v$ is said to be the \textbf{child} of $w$.
For a node $v \in \V$ the sets containing all its parents and children are denoted by $\pa{v}$ and $\ch{v}$, respectively. If there exists a path from $w$ to $v$, then $w$ is said to be an \textbf{ancestor} of $v$ and $v$ is said to be a \textbf{descendant} of $w$.
For a node $v \in V$ the sets containing all its ancestors and descendants are denoted by $\an{v}$ and $\de{v}$, respectively.

If a node has at least two parents, then we say that there is a \textbf{v-structure} at $v$ 
(also called \textbf{converging connection}) and when it has at least two children is referred to as a \textbf{diverging connection}.
Moreover, paths of the form $( v_1,v_2,v_3 )$ or $( v_3,v_2,v_1 )$ will be called \textbf{serial connections}.
Hence the following connections are of interest:
converging connection $v_1 \rightarrow v_2 \leftarrow v_3$;
serial connection $v_1 \rightarrow v_2 \rightarrow v_3$;
and diverging connection $v_1 \leftarrow v_2 \rightarrow v_3$.

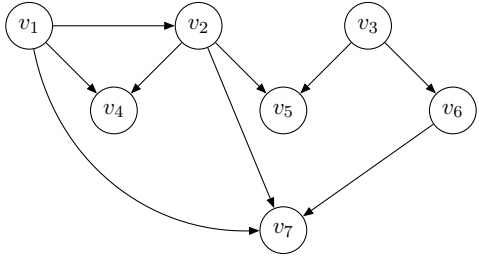
\begin{figure}[H]
    \centering
    \begin{tikzpicture}[scale=0.8, transform shape, node distance=1.2cm, state/.style={circle, draw=black}]
    
    \node[state] (v1) {$v_1$};
    \node[state, below right = of v1] (v4) {$v_4$};
    \node[state, above right = of v4] (v2) {$v_2$};
    \node[state, below right = of v2] (v5) {$v_5$};
    \node[state, above right = of v5] (v3) {$v_3$};
    \node[state, below right = of v3] (v6) {$v_6$};
    \node[state, below = of v5] (v7) {$v_7$};
    
    \begin{scope}[>={Stealth[length=4pt,width=3pt,inset=0pt]}]
    
    \path [->] (v1) edge node {} (v4);
     \path [->] (v1) edge node {} (v2);
    \path [->] (v1) edge[bend right= 40] node {} (v7);
    \path [->] (v2) edge node {} (v4);
    \path [->] (v2) edge node {} (v5);
    \path [->] (v2) edge node {} (v7);
    \path [->] (v3) edge node {} (v5);
    \path [->] (v3) edge node {} (v6);
    \path [->] (v6) edge node {} (v7);
    
    \end{scope}
    
    \end{tikzpicture}
  \caption{Directed acyclic graph with seven nodes.}
    \label{fig:DAG7}
\end{figure}
In this paper we restrict DAGs not to allow active cycle which is defined below and represented in \cref{fig:int_in_n_dim}.

\begin{definition}[Active cycle]
\label{def:active_cycle}
    Let $\G$ be a DAG.
    Consider a node $v \in \V$ with distinct parents $w, z\in \pa{v}$ which are connected by a trail $w \har x_1 \har \cdots \har x_n \har z$ satisfying the following conditions:
    \begin{enumerate}[label = (\roman*)]
        \item $n \geq 1$.
        \item $w \har x_1 \har \cdots \har x_n \har z$ consists of only diverging or serial connections.
        \item $v \leftarrow w \har x_1 \har \cdots \har x_n \har z \rightarrow v$ contains no chords.
    \end{enumerate}
    Then, the trail $v \leftarrow w \har x_1 \har \cdots \har x_n \har z \rightarrow v$ is called an \textbf{active cycle} in $\G$.
    Furthermore, $\G$ is said to contain an active cycle.
\end{definition}

\begin{minipage}{0.45\textwidth}
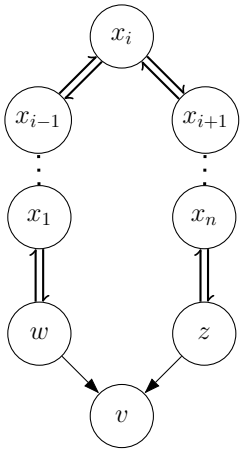
\begin{figure}[H]
    \centering
    \begin{tikzpicture}[scale=0.7, transform shape, node distance=1cm, state/.style={circle, font=\Large, draw=black, minimum size=1.2cm}]

    \node[state] (Un+1) {$v$};
    \node[state, above left = of Un+1] (U1) {$w$};
    \node[state, above right = of Un+1] (Un) {$z$};
    \node[state, above = 1cm of U1] (U2) {$x_1$};
    \node[state, above = 1cm of Un] (Un-1) {$x_{n}$};
    \node[state, above = 6cm of Un+1] (Ui) {$x_i$};
    \node[state, below left = of Ui] (Ui-1) {$x_{i-1}$};
    \node[state, below right = of Ui] (Ui+1) {$x_{i+1}$};

    \begin{scope}[>={Stealth[length=6pt,width=4pt,inset=0pt]}]

    \TikzHar{U1}{U2}{0.13em}{0.8pt}
    \TikzHar{Ui-1}{Ui}{0.13em}{0.8pt}
    \TikzHar{Ui}{Ui+1}{0.13em}{0.8pt}
    \TikzHar{Un-1}{Un}{0.13em}{0.8pt}
    
    \path [->] (U1) edge node {} (Un+1);
    \path [->] (Un) edge node {} (Un+1);

    \draw[loosely dotted, line width=0.4mm] (U2) -- (Ui-1);
    \draw[loosely dotted, line width=0.4mm] (Un-1) -- (Ui+1);
    
    \end{scope}
    \end{tikzpicture}
    \caption{Active cycle, where 
    $\har$ represents arcs that form only diverging or serial connections.}
    \label{fig:int_in_n_dim}
\end{figure}
\end{minipage}
\hspace{0.5cm}
\begin{minipage}{0.45\textwidth}
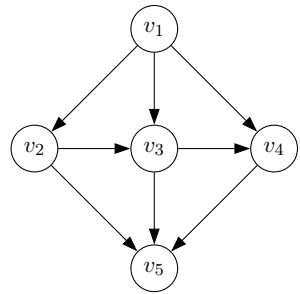
\begin{figure}[H]
    \centering
    \begin{tikzpicture}[scale=0.8, transform shape, node distance=1.2cm, state/.style={circle, draw=black}]
    \node [state] (v1) {$v_1$};
    \node[state, below= of v1] (v3) {$v_3$};
    \node[state, left= of v3] (v2) {$v_2$};
    \node[state, right= of v3] (v4) {$v_4$};
    \node[state, below= of v3] (v5) {$v_5$};
    \begin{scope}[>={Stealth[length=6pt,width=4pt,inset=0pt]}]
    \path [->] (v1) edge node {} (v2);
    \path [->] (v1) edge node {} (v3);
    \path [->] (v1) edge node {} (v4);
    \path [->] (v2) edge node {} (v3);
    \path [->] (v3) edge node {} (v4);
    \path [->] (v2) edge node {} (v5);
    \path [->] (v3) edge node {} (v5);
    \path [->] (v4) edge node {} (v5);
    \end{scope}
    \end{tikzpicture}
    \caption{A graph containing an active cycle. } \label{fig:active_cycle5}
\end{figure}
\end{minipage}

\medskip   

Note that the DAG in \cref{fig:DAG7} does not contain an active cycle even though its skeleton contains an undirected cycle $(v_7,v_2,v_5,v_3,v_6,v_7)$.
Indeed, there is a v-structure at $v_7$ whose parents are $v_2$ and $v_6$.
However, the trail $v_2 \har v_5 \har v_3 \har v_6 $ has a v-structure at node $v_5$.
This trail is of the form:
\begin{equation}
\label{eq:small_trail}
 v_2 \ra v_5 \la v_3 \har v_6. 
\end{equation}
Note that trail~\eqref{eq:small_trail} is of the form \eqref{eq:full_trail} with just one v-structure ($C=1$).
In \cref{fig:active_cycle5} a graph with an active cycle
$v_5 \leftarrow v_2 \leftarrow v_1 \rightarrow v_4 \rightarrow v_5$ is presented.

\medskip

An important concept in graphical models and in particular in BNs (whose qualitative part is represented by DAG) is that two subsets of nodes can be connected through trails.
These trails can be either blocked or activated given another subset of nodes \cite{Pearl}. 

\begin{definition}[d-separation]\label{def:graph_dsep}
    Let $\G=(\V,\E)$ be a directed graph and
    let $X, Y, Z\subseteq \V$ be disjoint and $X,Y$ nonempty sets.
    Then, $Z$ is said to \textbf{d-separate} $X$ and $Y$ in $\G$, denoted by $\dsepbig{X}{Y}{Z}$,
    if every trail $v_1 \har v_2 \har \cdots \har v_n$ with $v_1\in X$ and $v_n\in Y$ contains at least one node $v_i$ satisfying one of the following conditions:
    \begin{itemize}
        \item The trail forms a v-structure at $v_i$, i.e. 
        $v_{i-1} \rightarrow v_{i} \leftarrow v_{i+1}$,
        and the set $\{v_i\} \sqcup de(v_i)$ is disjoint from $Z$.
        \item The trail does not contain a v-structure at $v_i$ and $v_i\in Z$.
    \end{itemize}
    If a trail satisfies one of the conditions above, it is said to be \textbf{blocked} by $Z$, else it is \textbf{activated} by $Z$.
    Furthermore, if $X$ and $Y$ are not d-separated by $Z$, we use the notation $\notdsepbig{X}{Y}{Z}$.
\end{definition}

We can see that the trail~\eqref{eq:small_trail} is blocked by the empty set, because at $v_5$ there is a converging connection and this node does not belong to $Z=\emptyset$.
All other trails between $v_2$ and $v_6$ go through $v_7$ with converging connections.
Hence these trails are also blocked by $Z=\emptyset$ and we conclude that $\dsepbig{v_2}{v_6}{\emptyset}$.
It is not the case, however, that $\dsepbig{v_2}{v_6}{v_5}$ as the set $Z$ such that $v_5\in Z$ and $v_3 \notin Z$ activates trail~\eqref{eq:small_trail}.

\medskip

Note that the directed separation defined above is not equivalent to the concept of directed separation discussed in \cite{Erde2020}, where the separation concerns division of nodes in the graph into two subsets.
Our interest is in line with d-separation defined in \cite{Pearl} and discussed in \cite{Geiger_et_al_1990,KollerFriedman}.

\medskip

In the following sections, we will study properties of trails in DAGs without active cycles. The restriction we consider is one of the restrictions necessary 
for relatively efficient computations in the copula-based BNs introduced recently \cite{AHK_PCBN_2024}.

%% file: Ordered_Trails.tex
\section{Ordered set of trails}
\label{sec:setoftrails}
We start this section by introducing notation of general trails. Let us denote as  $X$, $Y$ and $Z$ three disjoint subsets of $\V$, where $Z$ is allowed to be empty.
First, we define the set of all trails from $X$ to $Y$ activated by $Z$ in $\G$.

\begin{definition}
    Let $X, Y, Z \subseteq \V$ be disjoint subsets of $\V$.
    Define $\TRAILSbig{X}{Y}{Z}$ to be the set of trails from $X$ to $Y$ activated by $Z$. 
\end{definition}

The set of all converging connections in a trail $T$ in $\TRAILSbig{X}{Y}{Z}$ is defined next.

\begin{definition}
\label{def:TRAILS_XYZ}
    For $T \in \TRAILSbig{X}{Y}{Z}$, the set
    $\ConvCon{T} := (c_1, \dots, c_C)$ is the ordered set of nodes corresponding to converging connections in $T$, ordered by first appearance on the trail from $X$ to $Y$.
    The trail $T$ is of the form \eqref{eq:full_trail}
    with $x \in X$ and $y \in Y$.
    The cardinality of the set $\ConvCon{T}$ is denoted by $C := C(T) = \Cardi{\ConvCon{T}}$.
\end{definition}

For a trail $T$ to be activated by $Z$, we must have that for all $i=1,\dots,C$, $c_i$ is in $Z$, or one of its descendants is in $Z$, see \cref{def:graph_dsep}.
To differentiate situations when a node $c_i$ is in $Z$ or when this node is not in $Z$ but its descendant is
we introduce the definition of \textbf{closest descendant}.

\begin{definition}[Closest descendant]
\label{def:closest_descendant}
    Let $T$ be a trail in $\TRAILSbig{X}{Y}{Z}$ and $i \in \{1,\dots, C(T) \}$.
    If $c_i \notin Z$, then its \textbf{closest descendant} in $Z$ is a node $Z(c_i) \in Z$ such that there exist a shortest path
    $$
    c_i \ra d^i_1 \ra \cdots \ra d^i_{\n{Z}{i}} \ra Z(c_i)
    $$
    with $d^i_j \notin Z$ for all $j = 1, \dots, \n{Z}{i}$.

    Such a path is referred to as a \textbf{descendant path} of $c_i$.
    Its nodes on the descendant path are denoted by the symbol ``$d$'' where a superscript $i$ indicates that $d^i_j$ lies on the descendant path of $c_i$, and the subscript $j$ indicates that it is the $j$-th node on this path.
    The length of the descendant path is formally denoted by $\n{Z}{i}$, but we will often simply write $n:=\n{Z}{i}$.
    If $c_i \in Z$, we also say that $c_i = Z(c_i)$.
    Finally, we use the conventions $d^i_0 := c_i$ and $d^i_{n + 1} := Z(c_i)$.
\end{definition}

A trail $T$ in $\TRAILSbig{X}{Y}{Z}$ can be seen as a concatenation of trails activated by the empty set.
For instance, consider the trail~\eqref{eq:full_trail},
then each trail $c_i \la t^{i}_1 \har \cdots \har t^{i}_{\n{t}{i}} \ra c_{i+1}$ is a trail between two nodes activated by the empty set.
Such trails are referred to as \textbf{subtrails}.

\begin{definition}[Subtrails]
\label{def:subtrails}
    Let $T$ be a trail in $\TRAILSbig{X}{Y}{Z}$.
    Suppose that $T$ takes the form
    $$
    x \har t^{0}_1 \har \cdots \har t^{0}_{\n{t}{0}} \ra 
    c_1 \la 
    \cdotslong \ra
    c_C \la t^{C}_1 \har \cdots \har t^{C}_{\n{t}{C}} \har y.
    $$
    The following are referred to as the \textbf{subtrails} of $T$:
    \begin{align*}  
        &x \har t^{0}_1 \har \cdots \har t^{0}_{\n{t}{0}} \ra c_1,\\
        &c_i \la t^{i}_1 \har \cdots \har t^{i}_{\n{t}{i}} \ra c_{i+1}, \text{ with } i \in \{1,\dots,C-1 \},\\
        &c_C \la t^{C}_1 \har \cdots \har t^{C}_{\n{t}{C}} \har y.
    \end{align*}
    The nodes on the subtrails are denoted by the symbol ``$t$'' where a superscript $i$ indicates that $t^i_j$ lies in between $c_i$ and $c_{i+1}$ with the conventions $c_0 := x$ and $c_{C+1}:=y$.
    The subscript indicates its location on the subtrail.
    The length of a subtrail is formally denoted by $\n{t}{i}$, but we will often simply write $n:=\n{t}{i}$.
    
    Furthermore, we use the conventions $c_{0} := x$, $c_{C+1} := y$, $t^i_0 := c_i$ and $t^i_{n + 1} := c_{i+1}$.
    Also, we denote the common ancestor among a trail between $c_i$ and $c_{i+1}$ by $t^i_m$, see \cref{def:common_ancestor}.
\end{definition}

Now, we define a partial order for the set $\TRAILSbig{X}{Y}{Z}$. This allows to compare trails in $\TRAILSbig{X}{Y}{Z}$ while taking into account their particular structure.

\begin{definition}[Smaller trail]
\label{def:better_trail}
    Let $T_1$ and $T_2$ belong to $\TRAILSbig{X}{Y}{Z}$.
    We say that $T_1$ is a \textbf{smaller} trail than $T_2$,
    denoted by $T_1 \smallerTRAIL T_2$,
    if one of the following conditions is satisfied:
    \begin{enumerate}
        \item $\Cardi{\ConvCon{T_1} \setminus Z}
        < \Cardi{\ConvCon{T_2} \setminus Z}$.
        \item 1) is an equality
        and $\Cardi{\ConvCon{T_1}} := C(T_1) < C(T_2) =: \Cardi{\ConvCon{T_2}}$.
        \item 1) and 2) are equalities and
        $\sumlim_{i=1}^{C(T_1)} \n{Z}{i}(T_1)
        < \sumlim_{i=1}^{C(T_2)} \n{Z}{i}(T_2)$.
        \item 1), 2) and 3) are equalities and
        $\sumlim_{i=0}^{C(T_1)} \n{t}{i}(T_1)
        < \sumlim_{i=0}^{C(T_2)} \n{t}{i}(T_2)$.
    \end{enumerate}
\end{definition}

Note that the order $\smallerTRAIL$ on the set $\TRAILSbig{X}{Y}{Z}$ is induced 
by the alphabetical order on the vector
$\Big(\Cardi{\ConvCon{T} \setminus Z},$  $\Cardi{\ConvCon{T}},$  
$\sum_{i=1}^{C(T_1)} \n{Z}{i}(T)$, 
$\sum_{i=0}^{C(T_1)} \n{t}{i}(T)
\Big)$, 
for $T \in \TRAILSbig{X}{Y}{Z}$.
Indeed, we first order trails by number of converging connections not in $Z$, then by number of converging connections, then by total length of descendant paths and finally by total length of the subtrails.
This means that a smaller trail satisfies the following conditions.

\begin{enumerate}[label = C\arabic*.]
    \item It is a trail from $X$ to $Y$ activated by $Z$. \label{cond:active_XYZ}
    \item It contains a smaller number of converging nodes not contained in $Z$. \label{cond:least_converging_not_Z}
    \item Under the restrictions above, it contains fewer converging connections. \label{cond:least_converging}
    \item Under the restrictions above, the paths from converging nodes not contained in $Z$ to its closest descendants are shorter.\label{cond:shortest_descendant_paths}
    \item Under the restrictions above, it is a shorter such trail. \label{cond:shortest}
\end{enumerate}

The shortest trail in the DAG in \cref{fig:DAG7} between $X=\{v_1\}$ and $Y=\{v_6\}$ activated by $Z=\{v_5\}$ is trail  $v_1 \ra v_2 \ra v_5 \la v_3 \ra v_6$.

\begin{remark}
    In general, $\smallerTRAIL$ is not a total order.
    For example, consider the graph defined by $v_1 \ra v_2$, $v_2 \ra v_4$,
    $v_1 \ra v_3$, $v_3 \ra v_4$ and $v_2 \ra v_3$. This graph is the diamond graph with an horizontal arc to avoid the active cycle.
    Note that both trails
    $T_1 := (v_1 \ra v_2 \ra v_4)$ and 
    $T_2 := (v_1 \ra v_3 \ra v_4)$ belong to $\TRAILSbig{v_1}{v_4}{\emptyset}$.
    They are not comparable since all 4 comparisons in~\cref{def:better_trail} are equalities.
    The trails $T_1$ and $T_2$ are actually both minimal elements in the partially ordered set 
    $\left(\TRAILSbig{v_1}{v_4}{\emptyset} \, , \, \smallerTRAIL \right)$.
    On the contrary, the trail
    $T_3 := (v_1 \ra v_2 \ra v_3 \ra v_4)$ belongs to $\TRAILSbig{v_1}{v_4}{\emptyset}$, but is not minimal. This is because
    $T_1 \smallerTRAIL T_3$ and $T_2 \smallerTRAIL T_3$.
\end{remark}

%% file: No_Converging.tex
\section{About trails with no converging connection}
\label{sec:proof_trail_with_no_converging}

In this section the results concerning sub-trails are included. 
First, a simple but interesting result  which states that trails with no converging connections are equivalent to trails activated by the empty set is presented.

\begin{lemma} \label{lemma:trail_activ_emptyset}
    A trail is activated by the empty set if and only if it does not contain a converging connection.
\end{lemma}
\begin{proof}
    This statement follows directly from the definition of d-separation, see \cref{def:graph_dsep}.
\end{proof}

If a trail $v_1 \har \cdots \har v_n$ contains no node with converging connection in $G$, then it must have at most one node with diverging connection.
An intuitive property of such a diverging node is that it is an ancestor of both end-points $v_1$ and $v_n$.
Therefore, we refer to it as a \textbf{common ancestor}.
Whenever a trail contains only serial connections, the common ancestor is defined to be the end-point to which the arrows point away from.

\begin{definition}[Common ancestor]
\label{def:common_ancestor}
    Let $\G$ be a DAG and let $x_0$ and $x_{n+1}$ be two nodes joined by a trail 
    \begin{equation}
\label{eq:simple_trail}
x_0 \har x_1 \har \cdots \har x_n \har x_{n+1}
\end{equation}
    with no converging connections.
    The \textbf{common ancestor} $x_m$ among this trail is defined as follows.
    \begin{itemize}
        \item If the trail contains only serial connections and $x_0$ is an ancestor of $x_{n+1}$, then $x_m = x_0$.
        \item If the trail contains only serial connections and $x_{n+1}$ is an ancestor of $x_{0}$, then $x_m = x_{n+1}$.
        \item If the trail contains a node with diverging connection, then $x_m$ is this node.
    \end{itemize}
\end{definition}
The common ancestor, $x_m$, is well-defined, since exactly one of the cases above holds.
From now on, in every figure the common ancestors will be displayed in the middle of a trail.
Therefore, the common ancestor will always be denoted with a subscript ``$m$'' which is an abbreviation for ``middle''.

It will be of interest for results concerning more general types of trails to consider trails between nodes, e.g. $x_0$ and $x_{n+1}$, for which all nodes (except $x_0$ and $x_{n+1}$) on the trail are included in a certain subset $K \subseteq \V$. In this case we say that the trail consists only of elements of $K$.
The nodes $x_0$ and $x_{n+1}$  (end-points of this trail) may or may not be in $K$.

\begin{definition}
    Let $\G=(\V, \E)$ be a DAG, let $K \subseteq \V$.
    We say that the trail~\eqref{eq:simple_trail} consists only of elements of $K$ if for all $i = 1, \dots, n$, $x_i \in K$.
\end{definition}

In the case of trails that do not contain converging connections the order $\smallerTRAIL$ becomes very simple.
If $T_1,T_2$ do not contain converging connections then $T_1\smallerTRAIL T_2$ whenever the number of nodes in $T_1$ is smaller than the number of nodes in $T_2$ (and both trails are not comparable if they have the same length).
Hence we can consider a shortest such trail.
In a few lemmas below we prove that a shortest trail satisfying a certain property also satisfies a second property.
Let us first formalize what is meant by a property of a trail.

\begin{definition}[Trail property]
\label{def:trail_property}
    Let $\G$ be a DAG containing a trail~\eqref{eq:simple_trail}.
    A \textbf{property} $\prop := \prop(x_0, \dots, x_{n+1})$ specifies the existence of certain arcs between the nodes on the trail.
    Here, we mean that $\prop$ states that $\E$ contains a certain set of arcs $\{x_i \ra x_j; \, i \in I, j \in J\}$ with $I,J \subseteq \{0, 1, \dots,  n+1\}$.

    For instance, the following are regarded as trail properties:
    \begin{itemize}
        \item The first arc of the trail points to the left; $x_0 \la x_1$.
        \item The $i$-th and $j$-th node on the trail are adjacent; $x_i \har x_j$.
        \item The trail is of the form $x_0 \la x_1 \ra x_2 \ra \cdots \ra x_{n-1} \ra x_n$, and we have that $x_0 \ra x_{n-1}$. 
    \end{itemize}
\end{definition}

Proofs where one property of a trail implies another will not only hold for shortest trails but also for shortest trails consisting of nodes in a subset $K \subseteq \V$. 
For instance, \cref{lemma:no_chords} also holds for shortest trails activated by the empty set and consisting of nodes in $K$.
Instead of repeatedly saying that a statement holds for both a shortest trail and a shortest trail consisting of nodes in a subset $K$ and proving both cases, we establish the following lemma.

\begin{lemma}
    \label{lemma:generalisation_to_K}
    For a trail $T$ of the form \eqref{eq:simple_trail}
    in DAG $\G$, let $\prop_1(x_0, x_1, \dots, x_{n+1})$ 
    and $\prop_2(x_0, x_1, \dots, x_{n+1})$
    be two trail properties.
    Let $\mathcal{G}$ be a set of DAGs such that
    \begin{itemize}
        \item for any DAG $\G = (\V, \E) \in \mathcal{G}$, for any $x_0, x_{n+1} \in V$,
        and for any shortest trail $T$ between $x_0$ and $x_{n+1}$ that satisfies $\prop_1$,
        the property $\prop_2$ holds.
        \item if $\G$ belongs to $\mathcal{G}$, then any graph obtained by removing vertices from $\G$ also belongs to $\mathcal{G}$.
    \end{itemize}
    
    If $\G = (\V, \E)$ is a DAG in $\mathcal{G}$ and $K \subseteq \V$, 
    then for any shortest trail between $x_0$ and $x_{n+1}$ that satisfies $\prop_1$ and that consists only of elements of $K$, the property $\prop_2$ still holds.
\end{lemma}

\begin{proof}
    Assume a trail $T$ of the form \eqref{eq:simple_trail} is a shortest trail between $x_0$ and $x_{n+1}$ that satisfies $\prop_1$ and consists only of elements of $K$.
    Consider the subgraph $\G^*$ induced by
    $\{x_0, x_{n+1}\} \cup K$.
    Note that this trail is a shortest trail satisfying $\prop_1$ between $x_0$ and $x_{n+1}$ in $\G^*$.
    Therefore, by assumption, it must satisfy $\prop_2$.
\end{proof}

\begin{remark}
    The class of restricted DAGs (containing no active cycle) satisfies the second assumption of \cref{lemma:generalisation_to_K}.
    Indeed, deleting nodes from a graph can never introduce an active cycle.
\end{remark}

In the next lemmas we will study properties of minimal trails without converging connections.
This means we assume that it is a shortest trail in terms of the number of nodes.
Knowing that it is shortest trail allows us to exclude the presence of certain chords.
For instance, a shortest trail activated by the empty set does not contain a chord.

\begin{lemma}
    \label{lemma:no_chords}
    Let $\G$ be a DAG with no active cycles 
    and let $T$ of the form \eqref{eq:simple_trail}
    be a trail in $\G$ for some $n \geq 0$.
    If this is a shortest trail between $x_0$ and $x_{n+1}$ activated by the empty set, then $T$ has no chords.
\end{lemma}
\begin{proof}
Let $x_m$ be the common ancestor of nodes in trail $T$, see \cref{def:common_ancestor}.

The proof is completed by remarking that the following connections are not possible:
\begin{itemize}
    \item $x_i \ra x_j$
    with $i < j \leq m$ results in a cycle.
    
    \item $x_i \la x_j$
    with $i < j \leq m$ results in a shorter trail.
    
    \item $x_i \ra x_j$
    with $m \leq i < j$ results in a shorter trail.
    
    \item $x_i \la x_j$
    with $m \leq i < j$ results in a cycle.
    
    \item $x_i \ra x_j$
    with $i < m < j$ results in a shorter trail.
    
    \item $x_i \la x_j$
    with $i < m < j$ results in a shorter trail.
\end{itemize}
\end{proof}

The lemma below states that if $\G$ contains a shortest trail of the form \eqref{eq:simple_trail} activated by the empty set for which $x_0 \ra v$ and $x_{n+1} \ra v$ for some node $v \in \V$, then for all $i=1,\dots,n$, $x_i \ra v$.

\begin{lemma}
\label{lemma:ac_all_parents}
    Let $\G$ be a DAG with no active cycles 
    and let $T$ of the form (\ref{eq:simple_trail})
    be a trail in $\G$ for some $n \geq 0$.
    If this is a shortest trail between $x_0$ and $x_{n+1}$ activated by the empty set, then
    \begin{itemize}
        \item[(i)] $\ch{x_0} \cap \ch{x_{n+1}} \subseteq
        \bigcap_{i = 1}^{n} \ch{x_i}$,
    
        \item[(ii)] $\forall i = 1, \dots, n,$
        $x_i \notin \ch{x_0} \cap \ch{x_{n+1}}$.
    \end{itemize}
\end{lemma}

\begin{proof}
    (ii) is a straightforward consequence of (i).
    We now prove (i).
    Let $v \in \ch{x_0} \cap \ch{x_{n+1}}$.
    To prove this, suppose that there exists an $i$ such that $v \notin \ch{x_i}$.
    We define the nodes $x_l$ and $x_r$ using the integers
    \begin{align*}
        l &:= \max \big\{j \in \{0, \dots, i-1 \}
        ;\, v \in \ch{x_j} \big\}, \\
        r &:= \min \big\{j \in \{i+1, \dots, n+1 \};
        \, v \in \ch{x_j} \big\}.
    \end{align*}
    With this notation, $x_l$ (respectively $x_r$) is the first node to the left (resp. right) of $x_i$ that is a parent of $v$.
    These integers $l$ and $r$ are well-defined since $v \in \ch{x_0} \cap \ch{x_{n+1}}$.
    Now, $\G$ contains the graph displayed in \cref{fig:bset_lemmahelp1}.
     \begin{figure}[H]
        \centering
        \begin{tikzpicture}[scale=0.8, transform shape, ,node distance=1.5cm, state/.style={circle, font=\Large, draw=black, minimum size=1cm}]
        \node (v) {$v$};
        \node[left of=v] (w1) {$x_0$};
        \node[right of=v] (w2) {$x_{n+1}$};
        \node[above left of=w1] (x1) {$x_1$};
        \node[above right of=w2] (xn) {$x_n$};
        \node[above of= x1] (xl) {$x_l$};
        \node[above of= xn] (xr) {$x_r$};
        \node[above right of=xl] (xi-1) {$x_{i-1}$};
        \node[above left of=xr] (xi+1) {$x_{i+1}$};
        \node[right of=xi-1] (xi) {$x_i$};
    
        \begin{scope}[>={Stealth[length=4pt,width=3pt,inset=0pt]}]
            \path [->] (w1) edge node[scale=0.8, above left] {} (v);
            \path [->] (w2) edge node[scale=0.8, above left] {} (v);
            \path [->, red] (xl) edge node[scale=0.8, above left] {} (v);
            \path [->, red] (xr) edge node[scale=0.8, above left] {} (v);
        \end{scope}

        \TikzHar{x1}{w1}{0.1em}{0.4pt}
        \TikzHar{w2}{xn}{0.1em}{0.4pt}
        {\color{red}\TikzHar{xi-1}{xi}{0.1em}{0.4pt}}
        {\color{red}\TikzHar{xi}{xi+1}{0.1em}{0.4pt}}
        
        
        
        \draw[loosely dotted, line width=0.4mm] (x1) -- (xl);
        \draw[red, loosely dotted, line width=0.4mm] (xl) -- (xi-1);
        \draw[loosely dotted, line width=0.4mm] (xn) -- (xr);
        \draw[red, loosely dotted, line width=0.4mm] (xr) -- (xi+1);
        
        
    
        \end{tikzpicture}
        \caption{Subgraph in $\G$ with the active cycle colored in red.}
        \label{fig:bset_lemmahelp1}
    \end{figure}
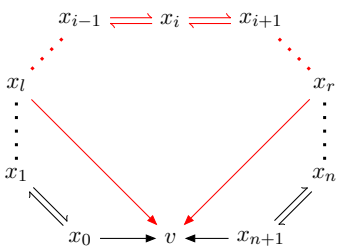
    
    Let us consider the trail 
    \begin{equation}
        v\la x_l \har \cdots \har x_i \har \cdots \har x_r \ra v.
        \label{trail:lemma:bset_btrails2}
    \end{equation}
    Any chord of this trail must be either a chord of $T$, an arc $v \ra x_j$ or an arc $x_j \ra v$ with
    $j\in \{l+1, \dots, r-1\}$.
    
    The first case is not possible by \cref{lemma:no_chords}. The second case is not possible because by \cref{lemma:trail_activ_emptyset} the trail $T$ contains at most one diverging connection, and therefore trail~(\ref{trail:lemma:bset_btrails2}) contains exactly one diverging connection.
    Consequently, any arc $v \ra x_j$ would result in a cycle.

    The third case is not possible by definition of $l$ and $r$.
    Therefore, we have shown that \eqref{trail:lemma:bset_btrails2} does not contain any chord.
    Thus, $\G$ contains the active cycle \eqref{trail:lemma:bset_btrails2}, which is a contradiction.
\end{proof}

It should be noted that we cannot use \cref{lemma:generalisation_to_K} to generalize the lemma above, because the properties ({\it i, ii}) in \cref{lemma:ac_all_parents} do not only concern the nodes $x_0, x_1, \dots, x_{n+1}$ but also their children.
Therefore, we prove the generalization in the corollary below.

\begin{corollary}
    \label{cor:ac_all_parents_general_K}
    Let $\G$ be a DAG with no active cycles 
    and let $T$ be a trail in $\G$ of the form \eqref{eq:simple_trail}.
    If this is a shortest trail between $x_0$ and $x_{n+1}$ activated by the empty set consisting of nodes in $K \subseteq \V$, then
    \begin{enumerate}[label=(\roman*)]
        \item $\ch{x_0} \cap \ch{x_{n+1}} \subseteq
        \bigcap_{i = 1}^{n} \ch{x_i}$;

        \item $\forall i = 1, \dots, n,$
        $x_i \notin \ch{x_0} \cap \ch{x_{n+1}}$.
    \end{enumerate}
\end{corollary}
\begin{proof}
    ({\it ii}) is a straightforward consequence of ({\it i}).
    We now prove ({\it i}).
    Trail~\eqref{eq:simple_trail} is a shortest trail activated by the empty set consisting of nodes in $K$, therefore by combining \cref{lemma:generalisation_to_K,lemma:no_chords} it contains no chords.
    
    Let $\G^* = (\V^*, \E^*)$ be the subgraph induced by 
    $$
    \V^* = K \cup \{x_0, x_{n+1} \} 
    \cup \big( \ch{x_0} \cap \ch{x_{n+1}} \big).
    $$
    By \cref{lemma:ac_all_parents}({\it ii}), any shortest trail between $x_0$ and $x_{n+1}$ in $\G^*$ activated by the empty set must not contain a node in $\ch{x_0} \cap \ch{x_{n+1}} \cap K = \ch{x_0} \cap \ch{x_{n+1}}$.
    Therefore, any shortest trail between $x_0$ and $x_{n+1}$ in $\G^*$ activated by the empty set consists of nodes in $K$.

    Thus, the trail $T$ is a shortest trail in $\G^*$ activated by the empty set.
    Now, we can apply \cref{lemma:ac_all_parents} to the trail $T$ in $\G^*$ to find that indeed 
    $\ch{x_0} \cap \ch{x_{n+1}} \subseteq \big( \cap_{i=1}^n \ch{x_i} \cap K \big) \subseteq \cap_{i=1}^n \ch{x_i}$.
\end{proof}

The lemma below states that if $v_1 \ra v_2$ for some $v_1,v_2 \in \V$, then the existence of a trail between $v_1$ and $v_2$ activated by the empty set and starting with an arc pointing to $v_1$ implies the existence of a particular subgraph.

\begin{theorem}
\label{lemma:po_active_cycle1}
    Let $\G$ be a DAG with no active cycles and let $v_1,v_2\in \V$ such that $v_1\ra v_2$.
    Suppose that 
      \begin{equation}\label{trail:subtrail2}
        v_1\la x_1 \har \cdots \har x_n \har v_2 
    \end{equation}
    is a shortest trail activated by the empty set starting with an arc $v_1 \la x_1$.
    Assume that $n \geq 1$.
    Then, for all $i \in \{1, \dots, n\}$,
    $x_i \ra x_{i+1}$ with the convention that $x_{n+1} := v_2$,
    and for all $i \in \{2, \dots, n\}$,
    $v_1 \ra x_{i}$ .
    
    This means that $\G$ contains the subgraph below.
    $$
    \begin{tikzpicture}[scale=0.8, transform shape, node distance=1.2cm, state/.style={circle, font=\Large, draw=black}]
    \node (v1) {$v_1$};
    \node[right= of v1] (v2) {$v_2$};
    \node[above left= of v1] (x2) {$x_2$};
    \node[above right= of v2] (xn-1) {$x_{n-1}$};
    \node[left = of x2] (x1) {$x_1$};
    \node[right = of xn-1] (xn) {$x_n$};

    \begin{scope}[>={Stealth[length=4pt,width=3pt,inset=0pt]}]
    \path [->] (v1) edge node {} (v2);
    \path [->] (v1) edge[bend right=0] node {} (x2);
    \path [->] (v1) edge node {} (xn-1);
    \path [->] (v1) edge node {} (xn);
    \path [->] (xn) edge node {} (v2);
    \path [->] (x1) edge node {} (v1);
    \path [->] (x1) edge node {} (x2);
    \path [->] (xn-1) edge node {} (xn);
    \draw[loosely dotted, line width=0.4mm] (x2) -- (xn-1);
    \end{scope}
    \end{tikzpicture}
    $$ 
    Furthermore, the theorem also holds for shortest trails activated by the empty set and of the form 
    \begin{equation}
    \label{trail:subtrail1}
        v_1\la x_1 \har \cdots \har x_n \ra v_2
    \end{equation}
    with $n\geq 1$.
\end{theorem}

\begin{proof}
    Consider a shortest trail of the form (\ref{trail:subtrail2}) activated by the empty set with $n \geq 1$.

    Consider the case when $n=1$.
    Here, the trail takes the form $v_1 \la x_1 \har v_2$ with $v_1 \ra v_2$.
    If $x_1 \la v_2$, then we obtain the cycle $v_1 \la x_1 \la v_2 \la v_1$, and therefore a contradiction.
    Consequently, the arc $x_1 \ra v_2$ must be present, giving us exactly the claimed subgraph, completing the proof.\\
    Now, let us assume that $n>1$.
    We first show that $x_n \ra v_2$.
    Suppose that $v_2 \ra x_n$, then the trail takes the form 
    $$
    v_1\la x_1 \har \cdots \har x_n \la v_2. 
    $$
    Since this trail is activated by the empty set it contains no converging connections (\cref{lemma:trail_activ_emptyset}).
    Hence, the trail must take the form
    $$
    v_1 \la x_1 \la \cdots \la x_n \la v_2.
    $$
    However, since $v_1 \ra v_2$, then we get a cycle and a contradiction.
    So, we get that $x_n \ra v_2$.\\
    Since it must be that $x_n \ra v_2$, the shortest trails of the form \eqref{trail:subtrail1} and of the form \eqref{trail:subtrail2} coincide.\\
    Let $x_m$ be the common ancestor in trail \eqref{trail:subtrail2}, see \cref{def:common_ancestor}. Now, $\G$ contains the subgraph in \cref{Fig:commdesc}.
   
    \begin{minipage}{0.45\textwidth}
    \begin{figure}[H]
      \centering
    \begin{tikzpicture}[scale=0.8, transform shape, node distance=0.7cm, state/.style={circle, font=\Large, draw=black}]
    \node (v1) {$v_1$};
    \node[right = of v2] (v2) {$v_2$};
    \node (help) at ($(v1)!0.5!(v2)$) {};
    \node[above = of help] (xm) {$x_m$};
    \node[left = of xm] (xm-1) {$x_{m-1}$};
    \node[left = of xm-1] (x2) {$x_2$};
    \node[left = of x2] (x1) {$x_1$};
    \node[right = of xm] (xm+1) {$x_{m+1}$};
    \node[right = of xm+1] (xn-1) {$x_{n-1}$};
    \node[right = of xn-1] (xn) {$x_n$};

    \begin{scope}[>={Stealth[length=4pt,width=3pt,inset=0pt]}]
    \path[->] (v1) edge node {} (v2);
    \path[->] (x1) edge node {} (v1);
    \path[->] (x2) edge node {} (x1);
    \path[->] (xm) edge node {} (xm-1);
    \path[->] (xm) edge node {} (xm+1);
    \path[->] (xn-1) edge node {} (xn);
    \path[->] (xn) edge node {} (v2);

    \draw[loosely dotted, line width=0.4mm] (x2) -- (xm-1);
    \draw[loosely dotted, line width=0.4mm] (xm+1) -- (xn-1);
    
    \end{scope}
    \end{tikzpicture}
    \caption{Subgraph of $G$ with common descendant.}\label{Fig:commdesc}
    \end{figure}
    \end{minipage}\hfill
    \begin{minipage}{0.45\textwidth}
    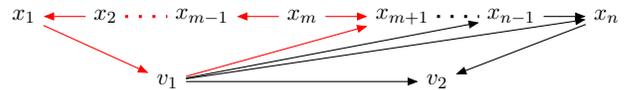
\begin{figure}[H]
      \centering
    \begin{tikzpicture}[scale=0.8, transform shape, node distance=0.7cm, state/.style={circle, font=\Large, draw=black}]
    \node (v1) {$v_1$};
    \node[right = of v2] (v2) {$v_2$};
    \node (help) at ($(v1)!0.5!(v2)$) {};
    \node[above = of help] (xm) {$x_m$};
    \node[left = of xm] (xm-1) {$x_{m-1}$};
    \node[left = of xm-1] (x2) {$x_2$};
    \node[left = of x2] (x1) {$x_1$};
    \node[right = of xm] (xm+1) {$x_{m+1}$};
    \node[right = of xm+1] (xn-1) {$x_{n-1}$};
    \node[right = of xn-1] (xn) {$x_n$};

    \begin{scope}[>={Stealth[length=4pt,width=3pt,inset=0pt]}]
    \path[->] (v1) edge node {} (v2);
    \path[->] (x1) edge[red] node {} (v1);
    \path[->] (x2) edge[red] node {} (x1);
    \path[->] (xm) edge[red] node {} (xm-1);
    \path[->] (xm) edge[red] node {} (xm+1);
    \path[->] (xn-1) edge node {} (xn);
    \path[->] (xn) edge node {} (v2);
    
    \path[->] (v1) edge[red] node {} (xm+1);
    \path[->] (v1) edge node {} (xn-1);
    \path[->] (v1) edge node {} (xn);

    \draw[loosely dotted, line width=0.4mm, red] (x2) -- (xm-1);
    \draw[loosely dotted, line width=0.4mm] (xm+1) -- (xn-1);
    
    \end{scope}
    \end{tikzpicture}
    \caption{Subgraph of $G$ with common descendant and chords.}\label{Fig:commdesccords}
\end{figure}
    \end{minipage}
    
    The subgraph above contains an undirected cycle with one converging connection (at $v_2$) hence the appropriate chords must be present. Several chords can be excluded:
    \begin{itemize}
        \item The trail $x_1 \har \cdots \har x_n \ra v_2$ is a shortest trail activated by the empty set, and therefore by \cref{lemma:no_chords} it has no chords.
        \item $v_1 \ra x_j$ with $j \leq m$ results in a cycle.
        \item $x_j \ra v_1$ with $j \in \{2, \dots, m\}$ results in a trail
        $v_1\la x_j \har \cdots \har x_n \ra v_2$
        which would be shorter than the shortest trail \eqref{trail:subtrail2} (while still being activated by the empty set).
        This is a contradiction.
    \end{itemize}
    The only remaining chords are of the form $v_1 \ra x_i$ with $i\in \{m+1,\dots,n \}$.
    First, we show that the diverging node $x_m$ must be the first node on trail~\eqref{trail:subtrail2}, i.e. $x_m = x_1$.
    To see this, we consider the case where all possible chords are present in $\E$, giving us the subgraph in \cref{Fig:commdesccords}.
    This graph contains an undirected cycle, coloured in red.
    
This undirected cycle is an active cycle, unless it is of length strictly smaller than 4.
    Thus, $\G$ must contain the subgraph as given by the theorem, completing the proof.    
\end{proof}

Similarly to the previous theorem, the theorem below states that under certain conditions the existence of a trail between two nodes $v_1$ and $v_2$ activated by the empty set implies the existence of a specific subgraph.
In this case, the conditions state that $v_1$ and $v_2$ are both parents of another node $v_3$ and the last arc along the trail between $v_1$ and $v_2$ points towards $v_2$. Moreover, no node on the trail can be a parent of $v_3$.

\begin{theorem}\label{lemma:v1_trail_to_v2_parents_of_v3}
    Let $\G$ be a DAG with no active cycles
    and let $v_1,v_2,v_3 \in \V$ such that $v_1, v_2\in \pa{v_3}$. 
    Suppose that $v_1$ and $v_2$ are connected by a trail
    \begin{equation}\label{trail:trail_start}
        v_1 \har x_1 \har \cdots \har x_n \ra v_2
    \end{equation}
    activated by the empty set with $\{x_i \}_{i=1}^n \cap \pa{v_3}=\emptyset$ and $n \geq 1$.
    If this is a shortest such trail, then $\G$ contains the subgraph below, with the convention $x_0:=v_1$.
    
   $$
    \begin{tikzpicture}[scale=0.8, transform shape, node distance=1.2cm, state/.style={circle, font=\Large, draw=black}]
    \node (v3) {$v_3$};
    \node[above left = of v3] (v1) {$v_1$};
    \node[above right = of v3] (v2) {$v_2$};
    \node (help) at ($(v1)!0.5!(v2)$) {};
    \node[above = of help] (xm) {};
    \node[left  = of xm] (xm-1) {};
    \node[left = of xm-1] (x2) {$x_2$};
    \node[left = of x2] (x1) {$x_1$};
    \node[right = of xm] (xm+1) {};
    \node[right = of xm+1] (xn-1) {$x_{n-1}$};
    \node[right = of xn-1] (xn) {$x_n$};

    \begin{scope}[>={Stealth[length=4pt,width=3pt,inset=0pt]}]
        \path [->] (v1) edge node[scale=0.8, above left] {} (v3);
         \path [->] (v2) edge node[scale=0.8, above left] {} (v1);
        \path [->] (v2) edge node[scale=0.8, above left] {} (v3);
        \path [->] (v2) edge node[scale=0.8, above left] {} (x2);
        \path [->] (v2) edge node[scale=0.8, above left] {} (xn-1);
        \path [->] (xn) edge node[scale=0.8, above left] {} (v2);
        \path [->] (x2) edge node {} (x1);
        \path [->] (xn) edge node {} (xn-1);
        \path [->] (x1) edge node {} (v1);

        \draw[loosely dotted, line width=0.4mm] (x2) -- (xn-1);
    \end{scope}
    \end{tikzpicture}
    $$
\end{theorem}
\begin{proof}
    Let us use the convention $x_{n+1} = v_2$.
    Let $x_m$ be the common ancestor of nodes in trail \eqref{trail:trail_start}, see \cref{def:common_ancestor}. Then, $\G$ contains the subgraph, that we will call $A$, displayed in \cref{Fig:commdescendant}.

    \begin{minipage}{0.45\textwidth}
    \begin{figure}[H]
      \centering
    \begin{tikzpicture}[scale=0.8, transform shape, node distance=0.7cm, state/.style={circle, font=\Large, draw=black}]
    \node (v3) {$v_3$};
    \node[above left = of v3] (v1) {$v_1$};
    \node[above right = of v3] (v2) {$v_2$};
    \node (help) at ($(v1)!0.5!(v2)$) {};
    \node[above = of help] (xm) {$x_m$};
    \node[left  = of xm] (xm-1) {$x_{m-1}$};
    \node[left = of xm-1] (x2) {$x_2$};
    \node[left = of x2] (x1) {$x_1$};
    \node[right = of xm] (xm+1) {$x_{m+1}$};
    \node[right = of xm+1] (xn-1) {$x_{n-1}$};
    \node[right = of xn-1] (xn) {$x_n$};
    
    \begin{scope}[>={Stealth[length=4pt,width=3pt,inset=0pt]}]
    \path [->] (v1) edge node[scale=0.8, above left] {} (v3);
    \path [->] (v2) edge node[scale=0.8, above left] {} (v3);
    \path [->] (xn) edge node[scale=0.8, above left] {} (v2);
    \path [->] (xm) edge node {} (xm+1);
    \path [->] (xm) edge node {} (xm-1);
    \path [->] (x2) edge node {} (x1);
    \path [->] (xn-1) edge node {} (xn);
    \path [->] (x1) edge node {} (v1);
    \end{scope}
    
    \draw[loosely dotted, line width=0.4mm] (x2) -- (xm-1);
    \draw[loosely dotted, line width=0.4mm] (xn-1) -- (xm+1);
    \end{tikzpicture}.
    \caption{Subgraph $A$ with common descendant.}\label{Fig:commdescendant}
    \end{figure}
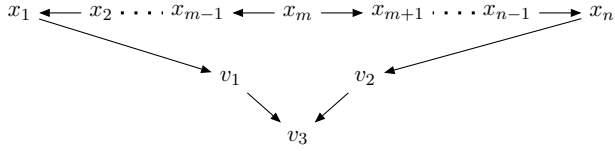
    \end{minipage}\hfill
     \begin{minipage}{0.45\textwidth}
    \begin{figure}[H]
      \centering
\begin{tikzpicture}[scale=0.8, transform shape, node distance=0.7cm, state/.style={circle, font=\Large, draw=black}]
    \node (v3) {$v_3$};
    \node[above left = of v3] (v1) {$v_1$};
    \node[above right = of v3] (v2) {$v_2$};
    \node (help) at ($(v1)!0.5!(v2)$) {};
    \node[above = of help] (xm) {$x_m$};
    \node[left  = of xm] (xm-1) {$x_{m-1}$};
    \node[left = of xm-1] (x2) {$x_2$};
    \node[left = of x2] (x1) {$x_1$};
    \node[right = of xm] (xm+1) {$x_{m+1}$};
    \node[right = of xm+1] (xn-1) {$x_{n-1}$};
    \node[right = of xn-1] (xn) {$x_n$};
    \begin{scope}[>={Stealth[length=4pt,width=3pt,inset=0pt]}]
        \path [->] (v1) edge node[scale=0.8, above left] {} (v3);
        \path [->] (v2) edge node[scale=0.8, above left] {} (v3);
        \path [->] (xn) edge[red] node[scale=0.8, above left] {} (v2);
        \path [->] (xm) edge[red] node {} (xm+1);
        \path [->] (xm) edge[red] node {} (xm-1);
        \path [->] (x2) edge node {} (x1);
        \path [->] (xn-1) edge[red] node {} (xn);
        \path [->] (x1) edge node {} (v1);
        \path [->] (v2) edge node {} (x1);
        \path [->] (v2) edge node {} (x2);
        \path [->] (v2) edge[red] node {} (xm-1);
        \path [->] (v2) edge node {} (v1);
    \end{scope}  
    \draw[loosely dotted, line width=0.4mm] (x2) -- (xm-1);
    \draw[loosely dotted, line width=0.4mm, red] (xn-1) -- (xm+1);
        \end{tikzpicture}
    \caption{Subgraph with common descendant and chords.}\label{Fig:commdescendantcords}
    \end{figure}
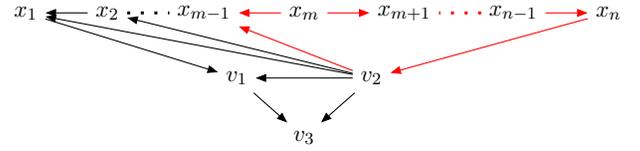
    \end{minipage}\hfill
    
    The graph $A$ after removing directions is a cycle. It has one converging connection (at $v_3$).
    Since $\G$ does not contain an active cycle, $A$ must contain the appropriate chords.
    Several chords can be excluded:
    \begin{itemize}
        \item The trail $v_1 \har x_1 \har \cdots \har x_n$ is a shortest trail activated by the empty set, and therefore it has no chords by \cref{lemma:no_chords}.
        \item $v_1 \ra v_2$ results in a shorter trail of the form \eqref{trail:trail_start}.
        \item $x_i \ra v_2$ results in a shorter trail of the form \eqref{trail:trail_start}.
        \item $v_2 \ra x_i$ with $i \geq m$ results in a cycle.
        \item $v_3 \ra x_i$ with $i=1,\dots,n$ results in a cycle.
        \item $x_i \ra v_3$ with $i=1,\dots,n$ cannot be present by the assumptions of the lemma.
    \end{itemize}
    Hence, the only possible chords are arcs of the form $v_2 \ra v_1$ and $v_2 \ra x_i$ with $i\in \{1,\dots,m-1 \}$.\\
First, we show that the common ancestor must be the last node in the trail, i.e. $x_m = x_n$.
    Consider the case where all possible chords are present in $A$, giving us the subgraph in \cref{Fig:commdescendantcords} with chords and undirected cycle containing one converging connection (at $x_{m-1}$) coloured in red. Since $\G$ cannot contain an active cycle, and there are no more arcs which could act as a chord, the length of this undirected cycle  must be strictly smaller than 4.
    Therefore $m = n$ and $x_m := x_n$ and we get that $\G$ contains the subgraph given by the lemma, completing the proof.
\end{proof}

In the next section more general types of trails are considered.
These are trails of the type (\ref{eq:full_trail}). The results that we presented in this section
will be applied to sub-trails of such more general trails.
The trails studied in \cref{lemma:po_active_cycle1} correspond to the first two sub-trails of trail~\eqref{eq:full_trail}, whereas the trails studied in \cref{lemma:v1_trail_to_v2_parents_of_v3} can be seen as sub-trails between the $c_i$'s.

%% file: With_converging.tex
\section{Properties of trails that may have converging connections}
\label{sec:proof:properties_trails_with_converging}

We will now prove some interesting properties of a minimal trail in $\TRAILSbig{X}{Y}{Z}$ with respect to $\smallerTRAIL$.

\begin{theorem}
\label{thm:minimal_trails}
    Let $X,Y,Z \subseteq \V$ be three disjoint subsets.
    Assume that $\TRAILSbig{X}{Y}{Z} \neq \emptyset$ and
    \begin{equation}
        x \har t^{0}_1 \har \cdots \har t^{0}_{\n{t}{0}} \ra 
        c_1 \la 
        \la \cdotslong \ra
        c_C \la t^{C}_1 \har \cdots \har t^{C}_{\n{t}{C}} \har y.
    \label{trail:lemma_XYZ_Cs}
    \end{equation}
    be a minimal element of $\TRAILSbig{X}{Y}{Z} $
    with respect to the order $\smallerTRAIL$.

    Then, the following properties hold:
    \begin{itemize}
        \item[(i)] For all $i,j$, $t^{i}_j \notin X \sqcup Y \sqcup Z$
        and $d^{i}_j \notin X \sqcup Y \sqcup Z$.
        \label{prop:nodes_along_subtrails}

        \item[(ii)] For all $i=1,\dots,C$, the trails $c_i \ra d^i_1 \ra \cdots \ra d^i_n \ra Z(c_i)$ and $t^i_1 \har \cdots \har t^i_n$ do not contain a chord.
        Furthermore, the trails $x \har t^0_1 \har \cdots \har t^0_n$ and
        $t^C_1 \har \cdots \har t^C_n \har y$ do not contain a chord.
        \label{prop:no_chords}

        \item[(iii)] If $c_i \ra c_{i+1}$ and $c_{i+1} \in Z$, then $c_i \in Z$.
        \label{prop:rightarrow_in_Z}
        
        \item[(iv)] If $c_i \la c_{i+1}$ and $c_{i} \in Z$, then $c_{i+1} \in Z$.
        \label{prop:leftarrow_in_Z}

        \item[(v)] For all $i=1,\dots,C-1$, the $i$-th subtrail is a shortest trail between $c_i$ and $c_{i+1}$ 
        starting with a leftward pointing arrow,
        ending with rightward pointing arrow,
        consisting of nodes in $\V \setminus Z$
        and with no converging connection.
        The $C$-th subtrail is a shortest trail between $c_C$ and $y$
        starting with a leftward pointing arrow,
        consisting of nodes in $\V \setminus Z$
        and with no converging connection.
        \label{prop:shortest_subtrails}
    \end{itemize}
\end{theorem}

\begin{proof}
    \textbf{\ref{prop:nodes_along_subtrails}(i):}
    We want to show that: for all $i,j$, 
    $t^{i}_j \notin Y \sqcup Z$ and 
    $d^{i}_j \notin Y \sqcup Z$.
    This will also show that $t^{i}_j \notin X$ and $d^{i}_j \notin X$ by symmetry.
    
    \medskip

    Assume that there exist $i,j$ such that $t^{i}_j \in Y \sqcup Z$.
    Remark first that $t^{i}_j \notin Z$,
    otherwise trail \eqref{trail:lemma_XYZ_Cs}
    would be blocked by $Z$
    (see \cref{def:graph_dsep}).
    Hence, $t^{i}_j$ must be in $Y$.
    Now, $x \har \cdots \har t^{i}_j$ is a trail from $x$ to an element of $Y$ activated by $Z$ that is smaller than \eqref{trail:lemma_XYZ_Cs} (see \cref{def:better_trail}). This contradicts the assumption that \eqref{trail:lemma_XYZ_Cs} is minimal.
    We have shown that
    $t^{i}_j \notin Y \sqcup Z$.

    Now, suppose that there exist $i,j$ such that $d^{i}_j \in Y \sqcup Z$.
    By \cref{def:closest_descendant}, this node cannot be in $Z$.
    Therefore, $d^{i}_j$ must be in $Y$.
    If this is the case, then the trail
    $$
    x \har \cdots \ra c_i \ra d^i_1 \ra \cdots \ra d^i_j
    $$
    would be a smaller trail in $\TRAILSbig{X}{Y}{Z}$ than \eqref{trail:lemma_XYZ_Cs}.
    Indeed, it contains at least one fewer converging node in $Z$, since the node $c_i$ now corresponds to a serial connection.
    This contradicts the assumption that \eqref{trail:lemma_XYZ_Cs} is minimal and concludes   the proof of \ref{prop:nodes_along_subtrails}(i).
    \medskip

    \textbf{\ref{prop:no_chords}(ii):}
    We want to show that, for all $i=1,\dots,C$, the trails
    $c_i \ra d^i_1 \ra \cdots \ra d^i_n \ra Z(c_i)$,
    $t^i_1 \har \cdots \har t^i_n$ and 
    $x \har t^0_1 \har \cdots \har t^0_n$ as well as
    $t^C_1 \har \cdots \har t^C_n \har y$ do not contain a chord.
    
    \medskip

    First, we consider a descendant path between $c_i$ and $Z(c_i)$ with $i \in \{1, \dots, C \}$.
    By \cref{def:closest_descendant}, this path is a shortest trail of the form
    $$
    c_i \ra d_1 \ra \cdots \ra d_n \ra Z(c_i)
    $$
    consisting of nodes in $\V \setminus Z$.
    By combining \cref{lemma:generalisation_to_K,lemma:no_chords}
    we know that this descendant path does not contain a chord.

    Now, let $i \in \{1, \dots, C \}$ and consider the subtrail
    $$
    c_i \la t_1 \har \cdots \har t_n \ra c_{i+1}.
    $$
    Observe that the trail $t_1 \har \cdots \har t_n$ is a shortest trail between $t_1$ and $t_n$ with no converging connections consisting of nodes in $\V \setminus \big( Z \sqcup Y \big)$.
    Indeed, if there would be a shorter such trail $T^*$ between $t_1$ and $t_n$, then replacing $t_1 \har \cdots \har t_n$ in \eqref{trail:lemma_XYZ_Cs} by $T^*$ would result in a smaller trail than \eqref{trail:lemma_XYZ_Cs}, and therefore a contradiction.
    Now, we can apply \cref{lemma:generalisation_to_K,lemma:no_chords}
    to find that $t_1 \har \cdots \har t_n$ cannot contain a chord.

    Consider the subtrail
    $$
    x \har t^0_1 \har \cdots \har t^0_n \ra c_1.
    $$
    By similar argument as above, the trail
    $x \har t^0_1 \har \cdots \har t^0_n$
    is a shortest trail activated by the empty set consisting of nodes in $V \setminus \big( Y \sqcup Z\big)$.
    Thus, we can apply \cref{lemma:generalisation_to_K,lemma:no_chords} to find that $x \har t^0_1 \har \cdots \har t^0_n$ contains no chords.

    The last trail does not contain a chord by symmetry: switch the role of $X$ and $Y$ and apply this result.
    This concludes the poof of \ref{prop:no_chords}(ii).
    \medskip

    \textbf{\ref{prop:rightarrow_in_Z}(iii):}
    We want to prove that, if $c_i \ra c_{i+1}$ and $c_{i+1} \in Z$, then $c_i \in Z$.
    \medskip
    
    Consider the case when $c_i \notin Z$, then the trail
    $
    x \har \cdots \ra c_i \ra c_{i+1} \la \cdots \har y
    $
    would be a smaller trail than \eqref{trail:lemma_XYZ_Cs} as it contains one fewer converging connections.
    This contradiction implies that we have $c_i \in Z$.
    \medskip

    \textbf{\ref{prop:leftarrow_in_Z}(iv):}
    This is a direct consequence of \ref{prop:rightarrow_in_Z}(iii) obtained by switching the roles of $X$ and $Y$.
    \medskip

    \textbf{\ref{prop:shortest_subtrails}(v):}
    We want to prove that, for all $i=1,\dots,C$, the trail $c_i \la t^i_1 \har t^{i}_2 \har \cdots \har t^{i}_{n-1} \har t^i_n \har c_{i+1}$ is a shortest such trail.
    \medskip

    This follows directly form the definition of \eqref{trail:lemma_XYZ_Cs}.
    If there would be a shorter trail $T^*$ between $c_i$ and $c_{i+1}$, then replacing the corresponding subtrail in \eqref{trail:lemma_XYZ_Cs} by $T^*$ would result in a smaller trail, and therefore a contradiction.    
\end{proof}

\medskip
In the further considerations we add an extra assumption on the subset $Y \sqcup Z$ of $\V$ and show additional properties of minimal trails in $\TRAILSbig{X}{Y}{Z}$ under this constraint on $Y$ and $Z$.
The additional assumption is motivated by the application of the results presented in this paper to copula-based BN models.
These models are restricted not to contain certain graphical structures which allow them to be computationally efficient \cite{AHK_PCBN_2024}.
In PCBNs the parents of each node $v\in \V$ are sorted in a particular manner.
This is equivalent to creating a sequence of ordered subsets of $\pa{v}$;
we require a sequence of sets
$$
\emptyset = K_0 \subsetneq K_1 \subsetneq \cdots \subsetneq K_{|\pa{v}|} = \pa{v},
$$
where for $i=1,\dots, m$, $|K_i| = i$.
For efficient computations in PCBNs this sequence has to be such that all ``relationships'' between the nodes in a subset $K_i$ are ``local''.
More specifically two nodes $v_1$ and $v_2$ are locally related in $K_i$ if they are adjacent (directly related) or if any active trail given the empty set between them consists of nodes in $K_i$ (indirectly locally related).

\begin{definition}
    \label{def:propadj}
    Let $\G$ be a DAG and $K$ a subset of $\V$.
    We say that $K$ has \textbf{local relationships} if
    for all $ v_1, v_2 \in K$ such that there exists a trail
    $$
    v_1 \har x_1 \har \cdots \har x_n \har v_2
    $$
    with $x_i \notin K$ for all $i=1, \dots, n$ and no converging connections, then $v_1$ and $v_2$ are adjacent.
\end{definition}

Obviously, $V$ and singletons always have local relationships.
In the following proposition, we show that any set that has local relationships can be decomposed into a partition of sets that have local relationships and that are pairwise d-separated.
Furthermore, each of these subsets can be chosen to be connected.

\begin{proposition}
    A set $K$ has local relationships if and only if there exists a partition $K = \mathop{\bigsqcup}\limits_{i=1}^k K_i$ such that each part $K_i$ is connected and has local relationships in $\G$, and $\forall i \neq j$, $\dsepbig{K_i}{K_j}{\emptyset}$.
\end{proposition}
\begin{proof}
    It is straightforward to see that $K = \mathop{\bigsqcup}\limits_{i=1}^k K_i$ has local relationships if the parts $K_i$ are connected and pairwise d-separated by the empty set.
    Indeed, let $v_1$ and $v_2$ in $K$ such that there exists a trail between them with no converging connection.
    Then $v_1$ and $v_2$ are not d-separated by the empty set; therefore they belong to the same $K_i$, which is assumed to have local relationships.

    \medskip
    
    We now prove the ``if'' part.
    Let $K_1, \dots, K_k$ be the partition of $K$ in equivalence classes for the equivalence relationship ``is connected in $K$ to''.
    By definition, each set $K_j$ is connected.

    \medskip

    First, we prove that each part $K_j$ has local relationships.
    Let $v_1$ and $v_2$ in $K_j$ for some $j$.
    Assume that there exists a trail $v_1 \har x_1 \har \cdots \har x_n \har v_2$
    with $x_i \notin K_j$ for all $i=1, \dots, n$ and no converging connections.
    Then we have $x_i \notin K$ for all $i=1, \dots, n$.
    Indeed, by contradiction, let $i$ be the smallest integer such that $x_i \in K$.
    Since $x_i \notin K_j$ we can define $\ell \neq j$ such that $x_i \in K_\ell$.
    Then we have two cases:
    \begin{enumerate}
        \item if $i = 1$, $v_1 \in K_j$ and $x_1 \in K_\ell$ are adjacent, which is impossible since the parts are the equivalence classes for the equivalence relationship ``is connected in $K$ to''.
        \item if $i > 1$, then $v_1 \har x_1 \har \cdots \har x_i$ is a trail with no converging connection between nodes in $K$ consisting of nodes in $\V \setminus K$.
        Because $K$ has local relationships, $v_1 \in K_j$ and $x_i \in K_\ell$ are adjacent.
        Again this not possible by the chosen partition.
    \end{enumerate}
    We have shown that $x_i \notin K$ for all $i=1, \dots, n$.
    Since $K$ has local relationships, $v_1$ and $v_2$ are adjacent.
    Therefore, we have shown that each part $K_j$ has local relationships.

    \medskip

    Let $K_i$ and $K_j$ such that $\notdsepbig{K_i}{K_j}{\emptyset}$, and $i \neq j$.
    Then there exist $v_1 \in K_i$, $v_2 \in K_j$ such that $\notdsepbig{v_1}{v_2}{\emptyset}$.
    Therefore, there exists a trail $v_1 \har x_1 \har \cdots \har x_n \har v_2$ with no converging connection. Let us pick such a trail between $K_i$ and $K_j$ of smallest length.
    We distinguish three cases:
    \begin{enumerate}
        \item $v_1 \in K_i$ and $v_2 \in K_j$ adjacent. This is impossible by the definition of our partition.        
        \item For all $i$, $x_i \notin K$. Therefore $v_1 \in K_i$, $v_2 \in K_j$ are adjacent, which is not possible by the same argument as above.
        \item There exists an $\ell$ such that $x_\ell \in K$.
        Let $\ell$ be the smallest integer such that $x_\ell \in K$.
        Consequently, because $K$ has local relationships, $v_1$ and $x_\ell$ are adjacent.
        Since $v_1 \in K_i$, we obtain that $x_\ell$ belongs to the connected component of $v_1$; i.e. to $K_i$.
        So we obtain a trail $x_\ell \har \cdots \har x_n \har v_2$ between $x_\ell \in K_i$ and $v_2 \in K_j$ that has no converging connection.
        This is a contradiction because this trail is strictly shorter than $v_1 \har x_1 \har \cdots \har x_n \har v_2$ which was chosen to be of minimal length.
    \end{enumerate}
    Therefore, we have proved that $\forall i \neq j$, $\dsepbig{K_i}{K_j}{\emptyset}$.
\end{proof}

\begin{remark}
    This is the best characterization of sets with local relationships.
    Indeed, there exist graphs with connected subsets that still do not have local relationships.
    For example, let us consider $\V = \{1, 2, 3, 4, 5\}$ with the edges
    $1 \rightarrow 2  \rightarrow 3  \rightarrow 4$ and
    $1 \rightarrow 5  \rightarrow 4$.
    Then $K = \{1, 2, 3, 4\}$ is connected but does not have local relationships because the trail $1 \rightarrow 5  \rightarrow 4$ has no converging connection but still $1$ and $4$ are not adjacent.
\end{remark}

\begin{corollary}
    Let $K$ be a set with local relationships. Then for every $v_1 \neq v_2 \in K$, $v_1$ and $v_2$ are either connected in $K$ or d-separated given the empty set.
\end{corollary}

\begin{remark}
    The results above give an explicit approach to construct examples of graph $\G = (\V, \E)$ with a subset $K$ that has local relationships;
    \begin{enumerate}
        \item choose an arbitrary DAG $(K, E_K)$,
        
        \item add other nodes and edges while respecting the principle: Do not add trails with no converging connections between nodes of $K$ that are not adjacent.
    \end{enumerate}
\end{remark}

The local relationship property can be lost by removal of one node.
Indeed, let $v \in V$.
Then $V \setminus \{v\}$ has local relationships if and only if the following conditions holds:
$\forall v_1, v_2 \in \pa{v} \sqcup \ch{v}$, if $(v_1, v, v_2)$ is a serial or diverging connection then $v_1$ and $v_2$ are adjacent.
In other words, all pairs of adjacent-to-$v$ nodes for which $v$ is a serial or diverging connection are adjacent to each other.
In particular, $V \setminus \{v\}$ always has local relationships if $v$ has no children.

In the theorem below we show that if we require the subset $Y \sqcup Z$ of $\V$ to have local relationships, then additional properties of the minimal trails can be proven.

\begin{theorem}
\label{thm:minimal_trails_local_relationship}
    Let $X,Y,Z \subseteq \V$ be three disjoint subsets and $Y \sqcup Z$ has local relationships (\cref{def:propadj}).
    Assume that $\TRAILSbig{X}{Y}{Z} \neq \emptyset$ and  let $T$ a trail of the form (\ref{trail:lemma_XYZ_Cs}) be a minimal element of $\TRAILSbig{X}{Y}{Z} $
    with respect to the order $\smallerTRAIL$. Then, the following properties hold.
    \begin{itemize}
        \item[(i)] The final converging node $c_C$ is in $Z$.
        \label{prop:final_c_in_Z}

        \item[(ii)] For all $i=1, \dots, C - 1$, we have $c_i \in Z$ or $c_{i+1} \in Z$.
        \label{prop:most_ci_in_Z}
        
        \item[(iii)] For all $i=1,\dots,C$, the nodes $c_i$ and $c_{i+1}$ are adjacent.
        \label{prop:adjacency}
        
        \item[(iv)]
        \label{prop:subgraphs}
        If this trail contains a total of $C>0$ converging connections, then $\G$ contains the subgraph below.
    
        \begin{figure}[H]
        \centering
        \begin{tikzpicture}[scale=1, transform shape, node distance=1cm, state/.style={circle, font=\Large, draw=black}]

        \node (c1) {$c_{1}$};
        \node[right= of c1] (c2) {$c_{2}$};
        \node[right = 1.5cm of c2] (cC-1) {$c_{C-1}$};
        \node[right= of cC-1] (cC) {$c_{C}$};
        \node[right = of cC] (y) {$y$};

        \node[left = of c1] (t0n) {$t^0_n$};
        \node[left = of t0n] (t01) {$t^0_1$};
        \node[left = of t01] (x) {$x$};

        \begin{scope}[>={Stealth[length=6pt,width=4pt,inset=0pt]}]
            \path [-] (c1) edge[bend left=45] node {} (c2);
            \path [-] (cC-1) edge[bend left=45] node {} (cC);
            \path [-] (cC) edge[bend left=45] node {} (y);
        
            \draw[transform canvas={yshift=0.21ex},-left to,line width=0.25mm] (c1) -- (c2);
            \draw[transform canvas={yshift=-0.21ex},left to-,line width=0.25mm] (c1) -- (c2);
        
            \draw[transform canvas={yshift=0.21ex},-left to,line width=0.25mm] (cC-1) -- (cC);
            \draw[transform canvas={yshift=-0.21ex},left to-,line width=0.25mm] (cC-1) -- (cC);
        
            \draw[transform canvas={yshift=0.21ex},-left to,line width=0.25mm] (cC) -- (y);
            \draw[transform canvas={yshift=-0.21ex},left to-,line width=0.25mm] (cC) -- (y);
        
            \draw[loosely dotted, line width=0.5mm] (c2) -- (cC-1);

            \path [->] (t0n) edge node {} (c1);
            \draw[loosely dotted, line width=0.5mm] (t01) -- (t0n);
            \draw[transform canvas={yshift=0.21ex},-left to,line width=0.25mm] (x) -- (t01);
            \draw[transform canvas={yshift=-0.21ex},left to-,line width=0.25mm] (x) -- (t01);
            
        \end{scope}
        \end{tikzpicture}
        \end{figure}
        Here, the curved lines represent one of the following two subgraphs.
        
        \begin{figure}[H]
        \centering
        \begin{subfigure}{0.45\linewidth}
        \centering
        \begin{tikzpicture}[scale=0.8, transform shape, node distance=1cm, state/.style={circle, font=\Large, draw=black}]
        \node (v1) {$c_i$};
        \node[right= of v1] (v2) {$c_{i+1}$};
        \node[above left= of v1] (x1) {$t^i_1$};
        \node[above right= of v2] (xn) {$t^i_n$};
        \node[above = of x1] (x2) {$t^i_2$};
        \node[above = of xn] (xn-1) {$t^i_{n-1}$};
        \begin{scope}[>={Stealth[length=4pt,width=3pt,inset=0pt]}]
        \path [->] (v1) edge node {} (v2);
        \path [->] (v1) edge[bend right=0] node {} (x2);
        \path [->] (v1) edge node {} (xn-1);
        \path [->] (v1) edge node {} (xn);
        \path [->] (xn) edge node {} (v2);
        \path [->] (x1) edge node {} (v1);
        \path [->] (x1) edge node {} (x2);
        \path [->] (xn-1) edge node {} (xn);
        \draw[loosely dotted, line width=0.4mm] (x2) -- (xn-1);
        \end{scope}
        \end{tikzpicture}
                \end{subfigure}
                \begin{subfigure}{0.45\linewidth}
                \centering
        \begin{tikzpicture}[scale=0.8, transform shape, node distance=1cm, state/.style={circle, font=\Large, draw=black}]
        \node (v1) {$c_i$};
        \node[right= of v1] (v2) {$c_{i+1}$};
        \node[above left= of v1] (x1) {$t^i_1$};
        \node[above right= of v2] (xn) {$t^i_n$};
        \node[above = of x1] (x2) {$t^i_2$};
        \node[above = of xn] (xn-1) {$t^i_{n-1}$};
        \begin{scope}[>={Stealth[length=4pt,width=3pt,inset=0pt]}]
        \path [->] (v2) edge node {} (v1);
        \path [->] (v2) edge[bend right=0] node {} (x2);
        \path [->] (v2) edge node {} (xn-1);
        \path [->] (v2) edge node {} (x1);
        \path [->] (xn) edge node {} (v2);
        \path [->] (x1) edge node {} (v1);
        \path [->] (x2) edge node {} (x1);
        \path [->] (xn) edge node {} (xn-1);
        \draw[loosely dotted, line width=0.4mm] (x2) -- (xn-1);
        \end{scope}
        \end{tikzpicture}
        \end{subfigure}
        \end{figure}
    \end{itemize}
\end{theorem}

\begin{proof}
We prove each property separately.

\medskip

    \textbf{\ref{prop:final_c_in_Z}(i):} We want to show that the final converging node $c_C$ is in $Z$.
    \medskip

    Consider the case when $c_C \neq Z(c_C)$.
    Then, $\G$ contains the trail 
    $$
    Z(c_C) \la d_n \la \cdots \la d_1 \la c_C \la t_1 \har \cdots \har t_n \har y.
    $$
    This is a trail between two nodes in $Y \sqcup Z$ consisting of nodes not in $Y \sqcup Z$ by \cref{prop:nodes_along_subtrails}(i).
    Since the trail does not contain any converging connections and $Y \sqcup Z$ has local relationships we find that $Z(c_C)$ and $y$ must be adjacent.
    Assume that the arc $Z(c_c) \la y$ is present.
    Consider the trail
    \begin{align*}
        x \har \cdotslong \ra c_C \ra d_1 \ra \cdots \ra d_n \ra Z(c_C) \la y.
    \end{align*}
    In this trail, $c_C$ is now not a converging connection, instead $Z(c_C)$ is a converging node.
    Therefore, this trail has the same amount of converging connections $C$,
    but one fewer converging node corresponding to a node not in $Z$ than $T$. 
    This is because $c_C \notin Z$ while $Z(c_C) \in Z$.
    So, the trail above is smaller than $T$.
    Since $T$ is assumed to be a minimal trail, we have a contradiction, and therefore $\E$ must contain the arc $Z(c_C) \ra y$, giving us the subgraph in \cref{Fig:yZ(c_C)}.

   \begin{minipage}{0.45\textwidth}
    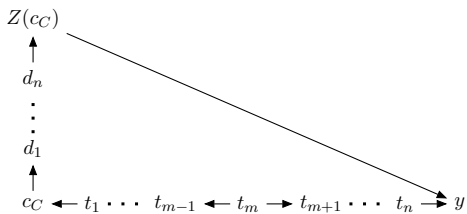
\begin{figure}[H]
        \centering
 \begin{tikzpicture}[scale=0.7, transform shape, node distance=0.5cm, state/.style={circle, font=\Large, draw=black}]
    \node (cj) {$c_C$};
    \node[above = of cj] (x1) {$d_1$};
    \node[above = 0.75cm of x1] (xn) {$d_n$};
    \node[above = of xn] (ocj) {$Z(c_C)$};
    \node[right = of cj] (y1) {$t_1$};
    \node[right = 0.75cm of y1] (ym-1) {$t_{m-1}$};
    \node[right = of ym-1] (ym) {$t_m$};
    \node[right = of ym] (ym+1) {$t_{m+1}$};
    \node[right = 0.75cm of ym+1] (yn) {$t_n$};
    \node[right = of yn] (oc) {$y$};
    
    \begin{scope}[>={Stealth[length=4pt,width=3pt,inset=0pt]}]
    
    \path [->] (cj) edge node {} (x1);
    \draw[loosely dotted, line width=0.4mm] (x1) -- (xn);
    \path [->] (xn) edge node {} (ocj);
    \path [->] (y1) edge node {} (cj);
    \draw[loosely dotted, line width=0.4mm] (y1) -- (ym-1);
    \path [->] (ym) edge node {} (ym-1);
    \path [->] (ym) edge node {} (ym+1);
    \draw[loosely dotted, line width=0.4mm] (ym+1) -- (yn);
    \path [->] (yn) edge node {} (oc);

    \path [->] (ocj) edge node {} (oc);

    \end{scope}
    \end{tikzpicture}
    \caption{Trail between $y$ and $Z(c_C)$.}\label{Fig:yZ(c_C)}
    \end{figure}
    \end{minipage}\;\;
    \begin{minipage}{0.45\textwidth}
\begin{figure}[H]
      \centering
 \begin{tikzpicture}[scale=0.7, transform shape, node distance=0.5cm, state/.style={circle, font=\Large, draw=black}]
    \node (cj) {$c_C$};
    \node[above = of cj] (x1) {$d_1$};
    \node[above = 0.75cm of x1] (xn) {$d_n$};
    \node[above = of xn] (ocj) {$Z(c_C)$};
    \node[right = of cj] (y1) {$t_1$};
    \node[right = 0.75cm of y1] (ym-1) {$t_{m-1}$};
    \node[right = of ym-1] (ym) {$t_m$};
    \node[right = of ym] (ym+1) {$t_{m+1}$};
    \node[right = 0.75cm of ym+1] (yn) {$t_n$};
    \node[right = of yn] (oc) {$y$};
    
    \begin{scope}[>={Stealth[length=4pt,width=3pt,inset=0pt]}]
    
    \path [->] (cj) edge[red] node {} (x1);
    \draw[loosely dotted, line width=0.5mm, red] (x1) -- (xn);
    \path [->] (xn) edge[red] node {} (ocj);
    \path [->] (y1) edge[red] node {} (cj);
    \draw[loosely dotted, line width=0.4mm, red] (y1) -- (ym-1);
    \path [->] (ym) edge[red] node {} (ym-1);
    \path [->] (ym) edge[red] node {} (ym+1);
    \draw[loosely dotted, line width=0.4mm] (ym+1) -- (yn);
    \path [->] (yn) edge node {} (oc);

    \path [->] (ocj) edge node {} (oc);
    \path [->] (ocj) edge node {} (yn);
    \path [->] (ocj) edge[red] node {} (ym+1);
    
    \end{scope}
    \end{tikzpicture}
        \caption{Trail between $y$ and $Z(c_C)$ with chords.}\label{Fig:yZ(c_C)withcords}
    \end{figure}
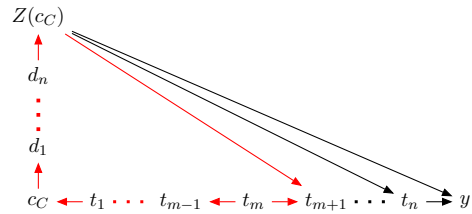        
\end{minipage}

\medskip    
    The undirected cycle above has one converging connection (at $y$); therefore it is an active cycle, unless $\E$ contains the appropriate chords.
    Several chords can be excluded:
    
    \begin{itemize}
        \item The trails 
        $c_C \ra d_1 \ra \cdots \ra d_n \ra Z(c_C)$
        and $t_1 \har \cdots \har t_n \ra y$
        do not contain any chords by \ref{prop:no_chords}(ii).
        \item $\forall j = 0, \dots, n+1,$
        $\forall l = 1, \dots, m,$
        $d_j \ra t_l$ results in a cycle.
        
        \item $\forall j = 0, \dots, n,$
        $\forall l = m+1, \dots, n+1,$
        $d_j \ra t_l$ results in a trail with fewer converging connections.
        
        \item $\forall j = 1, \dots, n,$
        $\forall l = 1, \dots, n+1,$
        $t_l \ra d_j$ results in a trail with shorter descendant paths.
        
        \item $\forall l = 2, \dots, n+1,$
        $t_l \ra c_C$ results in a shorter trail.
        
        \item $\forall l = 1, \dots, n,$
        $t_l \ra Z(c_C)$ results in a trail with fewer converging nodes not in $Z$.
    \end{itemize}
    
    Therefore, the only allowed chords are arcs from the node $Z(c_C)$ to nodes in $\{t_j\},$ $j=m+1, \dots, n$.
    The absence of any of them would result in an active cycle; therefore they all have to be present,
    giving us the subgraph in \cref{Fig:yZ(c_C)withcords}.
    The undirected cycle displayed in red is an active cycle, unless it is of length smaller than 4.
    It consists of the nodes $c_C$, $Z(c_C)$, $d_1, \dots ,d_n$ and $t_1,\dots,t_{m+1}$ and is therefore of length $2 + n + m+ 1 = n+m+3$.
    This means that $n+m+3 \leq 3$, and therefore $n+m=0$.
    However, this means that $t_m = t_0 := c_C$, and therefore $c_C \ra t_1$.
    This is a contradiction with the definition of $c_C$ since it is a converging node in $T$, which completes the proof of \ref{prop:final_c_in_Z}(i).

    \textbf{\ref{prop:most_ci_in_Z}(ii):}
    We want to show that, for all $i \in \{1, \dots, C - 1 \}$, we have $c_i \in Z$ or $c_{i+1} \in Z$.

    Assume that there exists an $i \in \{1, \dots, C - 1 \}$ such that $c_i, c_{i+1}\notin Z$.
    Then, we have the descendant paths
    \begin{align*}
        c_i \ra d^i_1 \ra \cdots \ra d^i_n \ra Z(c_i) \text{  and  }
        c_{i+1} \ra d^{i+1}_1 \ra \cdots \ra d^{i+1}_n \ra Z(c_{i+1}).
    \end{align*}
    Therefore,
    $Z(c_i)$ and $Z(c_{i+1})$ are two nodes in $Y \sqcup Z$ joined by a trail
    $$
    Z(c_i) \la \cdots \la c_i \la \cdots \ra c_{i+1} \ra \cdots \ra Z(c_{i+1})
    $$
    which is activated by the empty set and contains no nodes in $Y \sqcup Z$ (by \ref{prop:nodes_along_subtrails}(i)).
    Because $Y \sqcup Z$ has local relationships, the nodes $Z(c_i)$ and $Z(c_{i+1})$ must be adjacent.
    
    We can assume that $Z(c_i) \ra Z(c_{i+1})$, since the case $Z(c_i) \la Z(c_{i+1})$ follows by an analogous proof.
    Remark that $\G$ contains the subgraph in \cref{Fig:desc}.
    
    \begin{minipage}{0.45\textwidth}
    \begin{figure}[H]
        \centering
        \begin{tikzpicture}[scale=0.8, transform shape, node distance=0.6cm, state/.style={circle, font=\Large, draw=black}]

        \node (cj) {$c_i$};
        \node[right = of cj] (y1) {$t_1$};
        \node[right = of y1] (ym-1) {$t_{m-1}$};
        \node[right = of ym-1] (ym) {$t_m$};
        \node[right = of ym] (ym+1) {$t_{m+1}$};
        \node[right = of ym+1] (yn) {$t_n$};
        \node[right = of yn] (cj+1) {$c_{i+1}$};

        \node[above = 0.5cm of cj] (x1) {$d^i_1$};
        \node[above = of x1] (xn) {$d^i_{\n{Z}{i}}$};
        \node[above = 0.5cm of xn] (ocj) {$Z(c_i)$};

        \node[above = 0.5cm of cj+1] (z1) {$d^{i+1}_1$};
        \node[above = of z1] (zn) {$d^{i+1}_{n_Z(i +1)}$};
        \node[above = 0.5cm of zn] (ocj+1) {$Z(c_{i+1})$};

        \begin{scope}[>={Stealth[length=4pt,width=3pt,inset=0pt]}]
        \path [->] (ocj) edge node {} (ocj+1);

        \path [->] (ym) edge node {} (ym-1);
        \path [->] (ym) edge node {} (ym+1);
        \draw[loosely dotted, line width=0.4mm] (y1) -- (ym-1);
        \draw[loosely dotted, line width=0.4mm] (yn) -- (ym+1);
        \path [->] (y1) edge node {} (cj);
        \path [->] (yn) edge node {} (cj+1);
        
        \path [->] (cj) edge node {} (x1);
        \draw[loosely dotted, line width=0.4mm] (x1) -- (xn);
        \path [->] (xn) edge node {} (ocj);

        \path [->] (cj+1) edge node {} (z1);
        \draw[loosely dotted, line width=0.4mm] (z1) -- (zn);
        \path [->] (zn) edge node {} (ocj+1); 
        \end{scope}
        \end{tikzpicture}
        \caption{Graph with $Z(c_i)$ and $Z(c_{i+1})$. }\label{Fig:desc}
    \end{figure}
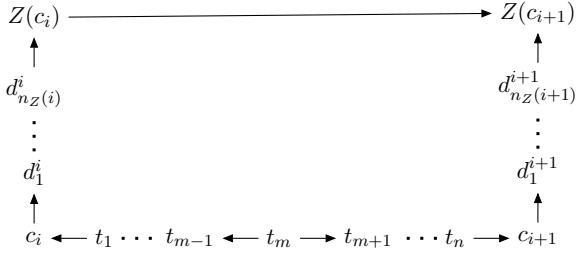
    \end{minipage}\hfill
    \begin{minipage}{0.45\textwidth}
    \begin{figure}[H]
        \centering
        \begin{tikzpicture}[scale=0.8, transform shape, node distance=0.6cm, state/.style={circle, font=\Large, draw=black}]

        \node (cj) {$c_i$};
        \node[right = of cj] (y1) {$t_1$};
        \node[right = of y1] (ym-1) {$t_{m-1}$};
        \node[right = of ym-1] (ym) {$t_m$};
        \node[right = of ym] (ym+1) {$t_{m+1}$};
        \node[right = of ym+1] (yn) {$t_n$};
        \node[right = of yn] (cj+1) {$c_{i+1}$};

        \node[above = 0.5cm of cj] (x1) {$d^i_1$};
        \node[above = of x1] (xn) {$d^i_{\n{Z}{i}}$};
        \node[above = 0.5cm of xn] (ocj) {$Z(c_i)$};

        \node[above = 0.5cm of cj+1] (z1) {$d^{i + 1}_1$};
        \node[above = of z1] (zn) {$d^{i+1}_{\n{Z}{i+1}}$};
        \node[above = 0.5cm of zn] (ocj+1) {$Z(c_{i+1})$};

        \begin{scope}[>={Stealth[length=4pt,width=3pt,inset=0pt]}]
        \path [->] (ocj) edge node {} (ocj+1);

        \path [->] (ym) edge[red] node {} (ym-1);
        \path [->] (ym) edge[red] node {} (ym+1);
        \draw[loosely dotted, line width=0.4mm, red] (y1) -- (ym-1);
        \draw[loosely dotted, line width=0.4mm] (yn) -- (ym+1);
        \path [->] (y1) edge[red] node {} (cj);
        \path [->] (yn) edge node {} (cj+1);
        
        \path [->] (cj) edge[red] node {} (x1);
        \draw[loosely dotted, line width=0.4mm, red] (x1) -- (xn);
        \path [->] (xn) edge[red] node {} (ocj);

        \path [->] (cj+1) edge node {} (z1);
        \draw[loosely dotted, line width=0.4mm] (z1) -- (zn);
        \path [->] (zn) edge node {} (ocj+1);

        \path [->] (ocj) edge node {} (zn);
        \path [->] (ocj) edge node {} (z1);
        \path [->] (ocj) edge node {} (cj+1);
        \path [->] (ocj) edge node {} (yn);
        \path [->] (ocj) edge[red] node {} (ym+1);
        
        \end{scope}
        
        \end{tikzpicture}
        \caption{Graph with $Z(c_i)$, $Z(c_{i+1})$ and chords.}\label{Fig:desccords}
    \end{figure}
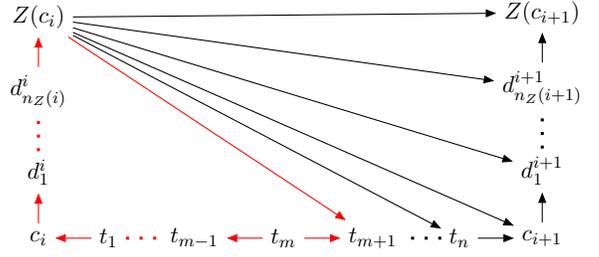
    \end{minipage}
    
    The undirected cycle above has one converging connection (at $Z(c_{i + 1}$); therefore it is an active cycle, unless $\E$ contains the appropriate chords.
    Several chords can be excluded:

    \begin{itemize}
        \item  The trails $c_i \ra d^i_1 \ra \cdots \ra d^i_n \ra Z(c_i)$, $t_1 \har \cdots \har t_n$ and $c_{i+1}\ra d^{i+1}_1 \ra \cdots \ra d^{i+1}_n \ra Z(c_{i+1})$ do not contain chords by \ref{prop:no_chords}(ii).

        \item $\forall j = 0, \dots, n_{Z}(i)+1$, $\forall l= 1, \dots, m$, $d^i_j \ra t_l$ results in a cycle.
        
        \item $\forall j = 0, \dots, n_{Z}(i)$, $\forall l= m+1, \dots, n+1$, $d^i_j \ra t_l$ results in a trail with fewer converging connections.

        \item $\forall j = 0, \dots, n_{Z}(i)$, $\forall l= 0, \dots, \n{Z}{i+1}+1$, $d^i_j \ra d^{i+1}_l$ results in a trail with fewer converging connections.

        \item $\forall j = 0, \dots, n_{Z}(i+1)+1$, $\forall l= m, \dots, n$, $d^{i+1}_j \ra t_l$ results in a cycle.
        
        \item $\forall j = 0, \dots, n_{Z}(i+1)$, $\forall l= 0, \dots, m-1$, $d^{i+1}_j \ra t_l$ results in a trail with fewer converging connections.

        \item $\forall j = 0, \dots, n_{Z}(i+1)$, $\forall l= 0, \dots, \n{Z}{i}+1$, $d^{i+1}_j \ra d^{i}_l$ results in a trail with fewer converging connections.

        \item $\forall j = 1, \dots, n_{Z}(i)$, $\forall l= 1, \dots, n$, $t_l \ra d^i_j$ results in a trail with shorter descendant paths (because $d^i_j$ becomes the new converging connection instead of $c_i$).

        \item $\forall j = 1, \dots, n_{Z}(i+1)$, $\forall l= 1, \dots, n$, $t_l \ra d^{i+1}_j$ results in a trail with shorter descendant paths (because $d^{i+1}_j$ becomes the new converging connection instead of $c_{i+1}$).

        \item $\forall l= 1, \dots, n$, $t_l \ra Z(c_{i})$ results in a trail with fewer converging nodes not in $Z$.
        
        \item $\forall l= 1, \dots, n$, $t_l \ra Z(c_{i+1})$ results in a trail with fewer converging nodes not in $Z$.

        \item $\forall l= 1, \dots, n$, $t_l \ra c_i$ and $t_l \ra c_{i+1}$ result in a shorter trail (the arcs $t_1 \ra c_1$ and $t_n \ra c_{i+1}$ are not chords).

        \item $\forall l= 0, \dots, n+1$, $Z(c_{i+1}) \ra t_l$ results in a cycle.

        \item $\forall j=0, \dots, \n{Z}{i} + 1$, $Z(c_{i+1}) \ra d^i_j$ results in a cycle.

    \end{itemize}
    Therefore, the only allowed chords are of the form 
    $Z(c_i) \ra t_l$ with $l \in \{m + 1, \dots, n+1 \}$ 
    and 
    $Z(c_i) \ra d^{i+1}_j$ with $j \in \{0, \dots, \n{Z}{i+1} \}$.
    All these arcs must be present to prevent an active cycle from occurring, giving us the subgraph in \cref{Fig:desccords}.

    This graph contains an undirected cycle with one converging connection (at $t_{m+1}$), coloured in red.
    There are no more chords which could be present.
    Therefore, this undirected cycle must be of length smaller than 4.
    The undirected cycle is made up of the nodes 
    $c_i$, $Z(c_i)$, $d^i_1, \dots, d^i_{\n{Z}{i}}$
    and
    $t_1, \dots, t_{m+1}$; it is of length
    $2 + \n{Z}{i} + m + 1 = \n{Z}{i} + m + 3$.
    This means that $\n{Z}{i} + m + 3 \leq 3$, and therefore $\n{Z}{i} = m = 0$.

    Thus, $t_m = t_0 := c_i$ must be the first diverging node on the subtrail between $c_i$ and $c_{i+1}$.
    However, this means that $c_i \ra t_1$.
    This is a contradiction with the definition of $c_i$ which is a converging connection in $T$.
    
    \textbf{\ref{prop:adjacency}(iii):}
    We want to show that for all $i=1,\dots,C$, the nodes $c_i$ and $c_{i+1}$ are adjacent.
    
    \medskip

    First, we consider the case when $i = C$.
    Here, $c_C$ is in $Z$ by \ref{prop:final_c_in_Z}(i), and $c_{C+1}:=y$ is in $Y$.
    Moreover, the nodes $c_C$ and $c_{C+1}$ are connected by the trail
    $$
    c_{C} \la t_1 \har \cdots \har t_n \har c_{C+1}
    $$
    with no converging connections and containing no nodes in $Y \sqcup Z$ by \cref{prop:nodes_along_subtrails}(i). From the fact that $Y \sqcup Z$ have local relationships we get that $c_C$ and $c_{C+1}$ are adjacent, completing the proof for the case when $i=C$.

    Now, we prove \ref{prop:adjacency}(iii) for $i \in \{1, \dots, C - 1 \}$.
    Note that, by~\ref{prop:most_ci_in_Z}(ii), at least one of the nodes $c_i$ and $c_{i+1}$ belongs to $Z$, giving us three cases.
    
    \underline{Case 1: $c_i \notin Z$ and $c_{i+1}\in Z$.}\\
    First, we remark that the arc $c_i \ra c_{i+1}$ is not possible by \cref{prop:rightarrow_in_Z}(iii).
    Therefore, we must show that $c_i \la c_{i+1}$.
    Suppose that this arc is not present in $\E$.
    This means that $c_i$ and $c_{i+1}$ are not adjacent.
    Furthermore, $\G$ contains the subgraph in \cref{Fig:subgraph}. 

\begin{minipage}{0.45\textwidth}
    \begin{figure}[H]
        \centering
    \begin{tikzpicture}[scale=0.6, transform shape, node distance=0.8cm, state/.style={circle, font=\Large, draw=black}]

   \node (cj) {$c_i$};
    \node[above = 0.5cm of cj] (x1) {$d^i_1$};
    \node[above = of x1] (xn) {$d^i_{\n{Z}{i}}$};
    \node[above = 0.5cm of xn] (ocj) {$Z(c_i)$};
    \node[right = of cj] (y1) {$t_1$};
    \node[right = of y1] (ym-1) {$t_{m-1}$};
    \node[right = of ym-1] (ym) {$t_m$};
    \node[right = of ym] (ym+1) {$t_{m+1}$};
    \node[right = of ym+1] (yn) {$t_n$};
    \node[right = of yn] (ocj+1) {$c_{i+1}$};

    \begin{scope}[>={Stealth[length=4pt,width=3pt,inset=0pt]}]
    
    \path [->] (cj) edge node {} (x1);
    \draw[loosely dotted, line width=0.4mm] (x1) -- (xn);
    \path [->] (xn) edge node {} (ocj);
    \path [->] (y1) edge node {} (cj);
    \draw[loosely dotted, line width=0.4mm] (y1) -- (ym-1);
    \path [->] (ym) edge node {} (ym-1);
    \path [->] (ym) edge node {} (ym+1);
    \draw[loosely dotted, line width=0.4mm] (ym+1) -- (yn);
    \path [->] (yn) edge node {} (ocj+1);
    \end{scope}
    
    \end{tikzpicture}
    \caption{Subgraph of $\G$.}\label{Fig:subgraph}
    \end{figure}
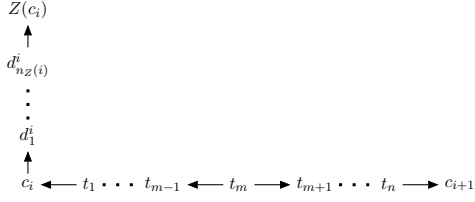
    \end{minipage}\hfill
    \begin{minipage}{0.45\textwidth}
\begin{figure}[H]
      \centering
 \begin{tikzpicture}[scale=0.6, transform shape, node distance=0.8cm, state/.style={circle, font=\Large, draw=black}]
     \node (cj) {$c_i$};
    \node[above = 0.5cm of cj] (x1) {$d^i_1$};
    \node[above = of x1] (xn) {$d^i_{\n{Z}{i}}$};
    \node[above = 0.5cm of xn] (ocj) {$Z(c_i)$};
    \node[right = of cj] (y1) {$t_1$};
    \node[right = of y1] (ym-1) {$t_{m-1}$};
    \node[right = of ym-1] (ym) {$t_m$};
    \node[right = of ym] (ym+1) {$t_{m+1}$};
    \node[right = of ym+1] (yn) {$t_n$};
    \node[right = of yn] (ocj+1) {$c_{i+1}$};

    \begin{scope}[>={Stealth[length=4pt,width=3pt,inset=0pt]}]
    
    \path [->] (cj) edge node {} (x1);
    \draw[loosely dotted, line width=0.4mm] (x1) -- (xn);
    \path [->] (xn) edge node {} (ocj);
    \path [->] (y1) edge node {} (cj);
    \draw[loosely dotted, line width=0.4mm] (y1) -- (ym-1);
    \path [->] (ym) edge node {} (ym-1);
    \path [->] (ym) edge node {} (ym+1);
    \draw[loosely dotted, line width=0.4mm] (ym+1) -- (yn);
    \path [->] (yn) edge node {} (ocj+1);

    \path [->] (ocj) edge node {} (ocj+1);
    
    \end{scope}
    \end{tikzpicture}
        \caption{Subgraph with $Z(c_i)\to c_{i+1}$.}\label{Fig:subgraph_arc1}
    \end{figure}        
\end{minipage}

\medskip
    Thus, $Z(c_i)$ and $c_{i+1} \in Z$ are joined by a trail
    $$
    Z(c_i) \la d_{\n{Z}{i}} \la \cdots \la d_1 \la c_i \la t_1 \har \cdots \har t_n \ra c_{i+1}
    $$
    which is activated by the empty set and consists of nodes not in $Y \sqcup Z$ by \ref{prop:nodes_along_subtrails}(i), and hence they are adjacent due to the local relationship property of $Y \sqcup Z$.
    We consider both cases; when $Z(c_i) \ra c_{i+1}$ and when $Z(c_i) \la c_{i+1}$.

    First, let us assume that $Z(c_i) \ra c_{i+1}$, giving us the subgraph in \cref{Fig:subgraph_arc1}.
    This subgraph contains an undirected cycle with one converging connection (at $c_{i+1}$), hence the appropriate chords must be present.
    The same arcs which provided a contradiction in the proof of \ref{prop:most_ci_in_Z}(ii) still do\footnote{This statement holds because $c_{i+1} = Z(c_{i+1})$}.
    This means that the only possible chords are $Z(c_i) \ra t_l$ with $l \in \{m+1, \dots, n\}$.
    It is evident that all such arcs are required to be present to prevent an active cycle, giving us the subgraph in \cref{Fig:subgraphcords}.
    
    \begin{minipage}{0.45\textwidth}
    \begin{figure}[H]
        \centering
    \begin{tikzpicture}[scale=0.6, transform shape, node distance=0.8cm, state/.style={circle, font=\Large, draw=black}]

    \node (cj) {$c_i$};
    \node[above = 0.5cm of cj] (x1) {$d_1$};
    \node[above = of x1] (xn) {$d_{\n{z}{i}}$};
    \node[above = 0.5cm of xn] (ocj) {$Z(c_i)$};
    \node[right = of cj] (y1) {$t_1$};
    \node[right = of y1] (ym-1) {$t_{m-1}$};
    \node[right = of ym-1] (ym) {$t_m$};
    \node[right = of ym] (ym+1) {$t_{m+1}$};
    \node[right = of ym+1] (yn) {$t_n$};
    \node[right = of yn] (ocj+1) {$c_{i+1}$};

    \begin{scope}[>={Stealth[length=4pt,width=3pt,inset=0pt]}]
    
    \path [->] (cj) edge[red] node {} (x1);
    \draw[loosely dotted, line width=0.4mm, red] (x1) -- (xn);
    \path [->] (xn) edge[red] node {} (ocj);
    \path [->] (y1) edge[red] node {} (cj);
    \draw[loosely dotted, line width=0.4mm, red] (y1) -- (ym-1);
    \path [->] (ym) edge[red] node {} (ym-1);
    \path [->] (ym) edge[red] node {} (ym+1);
    \draw[loosely dotted, line width=0.4mm] (ym+1) -- (yn);
    \path [->] (yn) edge node {} (ocj+1);

    \path [->] (ocj) edge node {} (ocj+1);
    \path [->] (ocj) edge[red] node {} (ym+1);
    \path [->] (ocj) edge node {} (yn);
    \end{scope}
    
    \end{tikzpicture}
    \caption{Subgraph of $\G$ with chords when $Z(c_i) \ra c_{i+1}$.}\label{Fig:subgraphcords}
    \end{figure}
    \end{minipage}\;\;
    \begin{minipage}{0.45\textwidth}
\begin{figure}[H]
      \centering
 \begin{tikzpicture}[scale=0.6, transform shape, node distance=0.8cm, state/.style={circle, font=\Large, draw=black}]
   \node (cj) {$c_i$};
    \node[above = 0.5cm of cj] (x1) {$d^i_1$};
    \node[above = of x1] (xn) {$d^i_{\n{Z}{i}}$};
    \node[above = 0.5cm of xn] (ocj) {$Z(c_i)$};
    \node[right = of cj] (y1) {$t_1$};
    \node[right = of y1] (ym-1) {$t_{m-1}$};
    \node[right = of ym-1] (ym) {$t_m$};
    \node[right = of ym] (ym+1) {$t_{m+1}$};
    \node[right = of ym+1] (yn) {$t_n$};
    \node[right = of yn] (ocj+1) {$c_{i+1}$};

    \begin{scope}[>={Stealth[length=4pt,width=3pt,inset=0pt]}]
    
    \path [->] (cj) edge[red] node {} (x1);
    \draw[loosely dotted, line width=0.4mm] (x1) -- (xn);
    \path [->] (xn) edge node {} (ocj);
    \path [->] (y1) edge[red] node {} (cj);
    \draw[loosely dotted, line width=0.4mm] (y1) -- (ym-1);
    \path [->] (ym) edge node {} (ym-1);
    \path [->] (ym) edge node {} (ym+1);
    \draw[loosely dotted, line width=0.4mm] (ym+1) -- (yn);
    \path [->] (yn) edge node {} (ocj+1);

    \path [->] (ocj+1) edge node {} (ocj);


    \path [->] (ocj+1) edge[bend left = 30, red] node {} (y1);
    \path [->] (ocj+1) edge[bend left = 30] node {} (ym-1);

    \path [->] (ocj+1) edge[red] node {} (x1);
    \path [->] (ocj+1) edge node {} (xn);
    
    \end{scope}
    \end{tikzpicture}
        \caption{Subgraph of $\G$ with chords when $Z(c_i) \la c_{i+1}$.}\label{Fig:subgraph_arc}
    \end{figure}        
\end{minipage}
    
\medskip

    This provides us with the same undirected cycle as displayed in the proof of \ref{prop:most_ci_in_Z}(ii), and therefore analogously we have a contradiction.

    Now, suppose that $Z(c_i) \la c_{i+1}$.
    As before, the undirected cycle is an active cycle, unless the appropriate chords are present.
    We can exclude several chords:
    \begin{itemize}
        \item  The trails $c_i \ra d^i_1 \ra \cdots \ra d^i_{\n{Z}{i}} \ra Z(c_i)$ and $t_1 \la \cdots \ra t_n$ do not contain chords by \ref{prop:no_chords}(ii).

        \item $\forall j = 0, \dots, n_{Z}(i)+1$, $\forall l= 1, \dots, m$, $d^i_j \ra t_l$ results in a cycle.

        \item $\forall j = 0, \dots, n_{Z}(i)$, $\forall l= m+1, \dots, n+1$, $d^i_j \ra t_l$ results in a trail with fewer converging connections.

        \item $\forall j = 1, \dots, n_{Z}(i)$, $\forall l= 1, \dots, n$, $t_l \ra d^i_j$ results in a trail with shorter descendant paths (because $d^i_j$ becomes the new converging connection instead of $c_i$).

        \item $\forall l= 1, \dots, n$, $t_l \ra c_i$ and $t_l \ra c_{i+1}$ results in a shorter trail whenever these are chords.

        \item $\forall l= 1, \dots, n$, $t_l \ra Z(c_i)$ result in a trail with fewer converging connections not in $Z$.

        \item $\forall l = m, \dots, n$, $Z(c_i) \ra t_l$ results in a cycle.

        \item $\forall l=m+1,\dots,n$, $c_i \ra t_l$ results in a cycle.
    \end{itemize}
    Therefore, the only possible chords are $c_{i+1} \ra d^i_j$ with $j \in \{1, \dots, \n{Z}{i} \}$, $c_{i+1} \ra t_l$ with $l \in \{1, \dots, m-1\}$ and $c_{i+1} \ra c_i$.

    We will now show that the arc $c_{i+1} \ra c_i$ must be present to prevent the occurrence of an active cycle.
    Consider the case where all possible chords are present except $c_{i+1} \ra c_i$, giving us the subgraph in \cref{Fig:subgraph_arc}. This subgraph contains an undirected cycle with one converging connection (at $d^i_1$).
    It is made up of the nodes $c_i$, $c_{i+1}$, $t_1$ and $d^i_1$; therefore it is of length 4.
    To prevent the occurrence of an active cycle it must have a chord.
    The only possible chord is the arc $c_{i+1} \ra c_i$, and hence this arc must be present.

    \underline{Case 2: $c_i \in Z$ and $c_{i+1} \notin Z$.}\\
    This case follows a by an analogous proof as the previous case.
    
    \underline{Case 3: $c_i, c_{i+1}\in Z$.}\\
    The nodes $c_i$ and $c_{i+1}$ are two nodes in $Y \sqcup Z$ joined by a trail
    $$
    c_i \la t_1 \har \cdots \har t_n \ra c_{i+1}
    $$
    with no converging connections and containing no nodes in $Y \sqcup Z$ by \ref{prop:nodes_along_subtrails}(i).
    Because $Y \sqcup Z$ has local relationships, we know that $c_i$ and $c_{i+1}$ are adjacent.

    Thus, for each case we have found that $c_i$ and $c_{i+1}$ must be adjacent, completing the proof of \ref{prop:adjacency}(iii) .

    \textbf{\ref{prop:subgraphs}(iv):}
    We want to show that for all $i=1,\dots,C$, $\G$ contains one of the considered two subgraphs. 
    
    \medskip
By \ref{prop:adjacency}(iii) we know that for all $i=1,\dots,C$, the nodes $c_i$ and $c_{i+1}$ are adjacent.
    Moreover, by \ref{prop:shortest_subtrails}(v), the trails
    $$
    c_i \la t^i_1 \har \cdots \har t^i_n \har c_{i+1},
    $$
    with $t^i_n \ra c_{i+1}$ if $i<C$, are shortest such trails consisting of nodes in $\V \setminus Z$.
    Therefore, we can apply \cref{lemma:generalisation_to_K,lemma:po_active_cycle1} to find that $\G$ contains one of the two subgraphs.
\end{proof}

In many simple cases, we can show that the converging nodes belong to $Z$, meaning that there are no descendant paths.
Below, two special cases where all the arrows point in the same direction are presented. In both of these cases it is shown that all the converging nodes $c_i$ are in $Z$.
Another simple case in the following corollary is when a converging node does not have a converging connection with the other converging nodes.

\begin{corollary}
\label{cor:all_cs_in_Z}
    Let us consider the setting of \cref{thm:minimal_trails_local_relationship}.
    \begin{itemize}
        \item[(i)] If the trail
        $c_1 \har \cdots \har  c_C$
        takes the form
        $c_1 \ra \cdots \ra  c_C$,
        then $\forall i=1,\dots,C$, $c_i \in Z$.
        \label{prop:rightarroww_all_in_Z}
        
        \item[(ii)] If $c_1 \in Z$ and the trail
        $c_1 \har \cdots \har c_C$
        takes the form
        $c_1 \la \cdots \la  c_C$,
        then $\forall i=1,\dots,C$, $c_i \in Z$.
        \label{prop:leftarrow_all_in_Z}
        
        \item[(iii)] Let $i \in \{2, \dots, C-1\}$.
        If the trail $c_{i-1} \har c_i \har c_{i+1}$ is not a converging connection, then $c_i \in Z$.
        \label{prop:no_convCon_c_i_in_Z}
    \end{itemize}
\end{corollary}

\begin{proof}
    The first part of this corollary is obtained by combining
    \cref{prop:final_c_in_Z}(i) and \cref{prop:rightarrow_in_Z}(iii).
    The second part of this corollary is a consequence of \cref{prop:leftarrow_in_Z}(iv).
    For the third part, combining Theorems~\ref{prop:most_ci_in_Z}(ii), \ref{prop:adjacency}(iii), \ref{prop:rightarrow_in_Z}(iii) and \ref{prop:leftarrow_in_Z}(iv) shows that $c_i \in Z$
    for the two cases
    $c_{i-1} \la c_i$ and $c_i \ra c_{i+1}$.
\end{proof}

\begin{remark}
    The condition in \cref{prop:no_convCon_c_i_in_Z} that $c_i$ is not a converging connection cannot be removed.
    Indeed, we now present a counter-example in which this condition is not satisfied.
    Consider the graph in \cref{fig:necessity_no_converging_c_i}, and let $X=\{x\}$, $Y=\{y\}$,
    and $Z=\{c_1, d_1, c_3\}$.
    Note that the trail $x \ra c_1 \la t_1 \ra c_2 \la t_2 \ra c_3 \la y$ is the minimal trail in $\TRAILSbig{X}{Y}{Z}$ since this is the only trail between $X$ and $Y$ activated by $Z$. Furthermore, it can be easily checked that $Y \sqcup Z$ has local relationships.
    Therefore, we are in the setting of \cref{thm:minimal_trails_local_relationship}, but still $c_2 \notin Z$.
    \begin{figure}[H]
        \centering
        \begin{tikzpicture}[scale=0.8, transform shape, node distance=1cm, state/.style={circle, draw=black}]

        \node [state] (x) {$x$};
        \node[state, right = of x] (c1) {$c_1$};
        \node[state, above = of c1] (t1) {$t_1$};
        \node[state, right = of c1] (c2) {$c_2$};
        \node[state, right = of c2] (c3) {$c_3$};
        \node[state, above = of c3] (t2) {$t_2$};
        \node[state, below = of c2] (d2) {$d_1$};
        \node[state, right = of c3] (y) {$y$};
        
        \begin{scope}[>={Stealth[length=4pt,width=3pt,inset=0pt]}]
        \path [->] (x) edge node {} (c1);
        \path [->] (c1) edge node {} (c2);
        \path [->] (c3) edge node {} (c2);
        \path [->] (y) edge node {} (c3);
        
        \path [->] (t1) edge node {} (c1);
        \path [->] (t1) edge node {} (c2);
        
        \path [->] (t2) edge node {} (c2);
        \path [->] (t2) edge node {} (c3);
        
        \path [->] (c1) edge node {} (d2);
        \path [->] (c2) edge node {} (d2);
        \path [->] (c3) edge node {} (d2);
        \end{scope}
        
        \end{tikzpicture}
        \caption{Graph illustrating the necessity of assumption in \cref{prop:no_convCon_c_i_in_Z}. }
    \label{fig:necessity_no_converging_c_i}
    \end{figure}
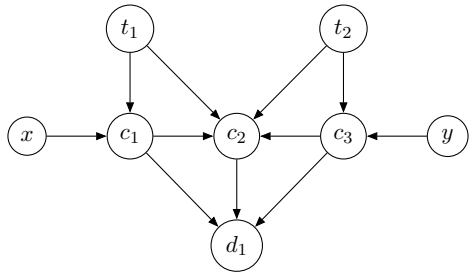

\end{remark}

%% file: Conclusion.tex
\section{Conclusion}
The trails that we considered in this paper were composed of distinct nodes as it is known that the existence of an active trail (with non-distinct elements) between two nodes in a DAG is equivalent with the existence of  an active trail (with distinct elements) between these nodes \cite{Geiger_et_al_1990}.

Our motivation to study properties of trails under specific conditions considered in this paper is  the application of these results in copula based Bayesian Networks.
However these results could also be of interest when searching for conditional independence that can be deduced from a given DAG.